\documentclass[11pt]{article}
\usepackage{amssymb,amsmath,accents}
\usepackage{amscd}
\usepackage{amsfonts,amsthm,mathrsfs}
\usepackage{setspace} 
\usepackage{graphics}
\usepackage{latexsym}
\usepackage{color}

\newtheorem{theorem}{Theorem}
\newtheorem{lemma}[theorem]{Lemma}
\newtheorem{proposition}[theorem]{Proposition}
\newtheorem{corollary}[theorem]{Corollary}

\theoremstyle{definition}

\title{\textbf{The Helmholtz decomposition of a space of vector fields with bounded mean oscillation in a bounded domain}}
\author{Yoshikazu Giga \\
The University of Tokyo \\
labgiga@ms.u-tokyo.ac.jp
\and
Zhongyang Gu \\
The University of Tokyo \\
zgu@ms.u-tokyo.ac.jp}

\begin{document}
\date{}

\maketitle
\begin{abstract}
We introduce a space of vector fields with bounded mean oscillation whose ``tangential'' and ``normal'' components to the boundary behave differently.
We establish its Helmholtz decomposition when the domain is bounded.
This substantially extends the authors' earlier result for a half space.
\end{abstract}
%
\begin{center}
\end{center}

%
\section{Introduction} 
\label{intro}
The Helmholtz decomposition of a vector field is a fundamental tool to analyze the Stokes and the Navier-Stokes equations. It is formally a decomposition of a vector field $v=(v^1,\ldots,v^n)$ in a domain $\Omega$ of $\mathbf{R}^n$ into
\begin{align} \label{H}
	v = v_{0} + \nabla q;
\end{align}
here $v_{0}$ is a divergence free vector field satisfying supplemental conditions like boundary condition and $\nabla q$ denotes the gradient of a function (scalar field) $q$. If $v$ is in $L^{p}$ ($1<p<\infty$) in $\Omega$, such a decomposition is well-studied.
 For example, a topological direct sum decomposition
\[
	\left( L^p(\Omega) \right)^n = L^p_\sigma(\Omega) \oplus G^p(\Omega)
\]
holds for various domains including $\Omega=\mathbf{R}^n$, a half space $\mathbf{R}^n_+$, a bounded smooth domain \cite{FM};
 see e.g.\ G.\ P.\ Galdi \cite{Ga}.
 Here, $L^p_\sigma(\Omega)$ denotes the $L^p$-closure of the space of all div-free vector fields compactly supported in $\Omega$ and $G^p(\Omega)$ denotes the totality of $L^p$ gradient fields.
 It is impossible to extend this Helmholtz decomposition to $L^\infty$ even if $\Omega=\mathbf{R}^n$ since the projection $v\mapsto\nabla q$ is a composite of the Riesz operators which is not bounded in $L^\infty$.
 We have to replace $L^\infty$ with a class of functions of bounded mean oscillation.
 However, if the vector field is of bounded mean oscillation ($BMO$ for short), such a problem is only studied when $\Omega$ is a half space $\mathbf{R}^n_+$ \cite{GigaGu}, where the boundary is flat.

Our goal is to establish the Helmholtz decomposition of $BMO$ vector fields in a smooth bounded domain in $\mathbf{R}^n$, which is a typical example of a domain with curved boundary.
 Although the space of $BMO$ functions in $\mathbf{R}^n$ is well studied, the situation is less clear when one considers such a space in a domain, because there are several possible definitions.
 One should be careful about the behavior of a function near the boundary $\Gamma = \partial\Omega$.
 In this paper we study a space of $BMO$ vector fields introduced in \cite{GigaGu2} and establish its Helmholtz decomposition when $\Omega$ is a bounded $C^3$ domain.

Let us recall the space $vBMO(\Omega)$ introduced in \cite{GigaGu2}.
 We first recall the $BMO$ seminorm for $\mu\in(0,\infty]$.
 For a locally integrable function $f$, i.e., $f \in L^1_\mathrm{loc}(\Omega)$ we define
\[
	[f]_{BMO^\mu(\Omega)} := \sup \left\{ \frac{1}{\left|B_r(x)\right|} \int_{B_r(x)} \left| f(y) - f_{B_r(x)} \right| \, dy \biggm|
	B_r(x) \subset \Omega,\ r < \mu \right\},
\]
where $f_B$ denotes the average over $B$, i.e.,
\[
	f_B := \frac{1}{|B|} \int_B f(y) \, dy
\]
and $B_r(x)$ denotes the closed ball of radius $r$ centered at $x$ and $|B|$ denotes the Lebesgue measure of $B$.
 The space $BMO^\mu(\Omega)$ is defined as
\[
	BMO^\mu(\Omega) := \left\{ f \in L^1_\mathrm{loc}(\Omega) \bigm|
	[f]_{BMO^\mu} < \infty \right\}.
\]
This space may not agree with the space of restrictions $r_\Omega f$ of $f \in BMO^\mu(\mathbf{R}^n)$.
 As in \cite{BG}, \cite{BGMST}, \cite{BGS}, \cite{BGST} we introduce a seminorm controlling the boundary behavior.
 For $\nu \in (0,\infty]$, we set
\[
	[f]_{b^\nu} := \sup \left\{ r^{-n} \int_{\Omega\cap B_r(x)} \left| f(y) \right| \, dy \biggm|
	x \in \Gamma,\ 0<r<\nu \right\}.
\]
In these papers, the space
\[
	BMO^{\mu,\nu}_b(\Omega) := \left\{ f \in BMO^\mu(\Omega) \bigm|
	[f]_{b^\nu} < \infty \right\}
\]
is considered.
 Note that this space $BMO^{\infty,\infty}_b(\Omega)$ is identified with Miyachi's $BMO$ introduced by \cite{Mi} if $\Omega$ is a bounded Lipschitz domain or a Lipschitz half space as proved in \cite{BGST}.
 However, unfortunately, it turns out such a boundary control for whole components of vector fields is too strict to have the Helmholtz decomposition.
 We separate tangential and normal components.
 Let $d_\Gamma(x)$ denote the distance from the boundary $\Gamma$, i.e., 
\[
	d_\Gamma(x) := \inf \left\{ |x - y|,\ y \in \Gamma \right\}.
\]
For vector fields, we consider
\[
	vBMO^{\mu,\nu}(\Omega) := \left\{ v \in \left(BMO^\mu(\Omega)\right)^n \bigm|
	[\nabla d_\Gamma \cdot v]_{b^\nu} < \infty \right\},
\]
where $\, \cdot \,$ denotes the standard inner product in $\mathbf{R}^n$.
The quantity $(\nabla d_\Gamma \cdot v)\nabla d_\Gamma$ on $\Gamma$ is the component of $v$ normal to the boundary $\Gamma$.
 We set
\[
	[v]_{vBMO^{\mu,\nu}(\Omega)} := [v]_{BMO^\mu(\Omega)} + [\nabla d_\Gamma \cdot v]_{b^\nu}.
\]
If $\Omega$ is the half space, this is not a norm but a seminorm.
 However, if it has a fully curved part in the sense of \cite[Definition 7]{GigaGu2}, then this becomes a norm \cite[Lemma 8]{GigaGu2}.
 In particular, when $\Omega$ is a bounded $C^2$ domain, this is a norm. 
 Roughly speaking, the boundary behavior of a vector field $v$ is controlled for only normal part of $v$ if $v \in vBMO^{\mu,\nu}(\Omega)$.
 For a bounded domain, this norm is equivalent no matter how $\mu$ and $\nu$ are taken; 
in other words, $vBMO^{\mu,\nu}(\Omega)=vBMO^{\infty,\infty}(\Omega)$.
 This is because $vBMO^{\mu,\nu}(\Omega)\subset L^1(\Omega)$ when $\Omega$ is bounded, which follows from the characterization of $vBMO^{\mu,\nu}(\Omega)$ in \cite[Theorem 9]{GigaGu2}.
 We shall simply write $vBMO^{\mu,\nu}(\Omega)$ as $vBMO(\Omega)$.
 We are now in a position to state our main result.
\begin{theorem} \label{M}
Let $\Omega$ be a bounded $C^3$ domain in $\mathbf{R}^n$.
 Then the topological direct sum decomposition
\begin{equation} \label{HD}
	vBMO(\Omega) = vBMO_\sigma(\Omega) \oplus 
	GvBMO(\Omega)
\end{equation}
holds with
\begin{align*}
	vBMO_\sigma(\Omega) &:= \left\{ v \in vBMO(\Omega) \bigm|
	\operatorname{div}v = 0\ \text{in}\ \Omega,\ 
	v \cdot \mathbf{n} = 0\ \text{on}\ \Gamma \right\}, \\
	GvBMO(\Omega) &:= \left\{ \nabla q \in vBMO(\Omega) \bigm|
	q \in L^1_\mathrm{loc}(\Omega) \right\},
\end{align*}
where $\mathbf{n}$ denotes the exterior unit normal vector field.
 In other words, for $v\in vBMO(\Omega)$, there is unique $v_0\in vBMO_\sigma(\Omega)$ and $\nabla q\in GvBMO(\Omega)$ satisfying $v=v_0+\nabla q$.
 Moreover, the mapping $v\mapsto v_0$, $v\mapsto\nabla q$ is bounded in $vBMO(\Omega)$.
\end{theorem}
%

As shown in \cite{GigaGu2}, the norm trace $v \cdot \mathbf{n}$ is well defined as an element of $L^\infty(\Gamma)$ for $v \in vBMO(\Omega)$ with $\operatorname{div}v = 0$.
 So far, the Helmholtz decomposition $BMO$ type space in a domain is only known for $vBMO^{\infty,\infty}$ when $\Omega$ is the half space
\[
	\mathbf{R}^n_+ = \left\{ x = (x_1, \ldots, x_n) \in \mathbf{R}^n \bigm|
	x_n > 0 \right\}
\] 
as shown in \cite{GigaGu}, where the normal trace is taken in locally $H^{-1/2}$ sense.

Here is our strategy to show Theorem \ref{M}.
 For a vector field $v$, we construct a linear map $v\longmapsto q_1$ such that $q_1$ satisfies
\[
	- \Delta q_1 = \operatorname{div}v \quad\text{in}\quad \Omega,
\]
where the divergence is taken in the sense of distribution.
There are many ways to construct such a map because there is no boundary condition.
A naive way is to extend $v$ in a suitable way to a function $\overline{v}$ on $\mathbf{R}^n$ so that $v\longmapsto\overline{v}$ is linear.
 We next consider the volume potential of $\operatorname{div}\overline{v}$, i.e.,
\[
	q_0(x) := \int_{\mathbf{R}^n} E(x-y) \operatorname{div} \overline{v}(y) \, dy 
	= E * \operatorname{div} \overline{v},
\]
where $E$ is the fundamental solution of $-\Delta$ in $\mathbf{R}^n$, i.e.,
\[
	E(x) :=
	\begin{cases}
  - \log |x| / 2 \pi & \quad (n = 2) \\
  |x|^{2-n} / \left( n(n-2) \alpha(n) \right) & \quad (n \geq 3),
  \end{cases}
\]
where $\alpha(n)$ denotes the volume of the unit ball $B_1(0)$ of $\mathbf{R}^n$.
 By the famous $BMO$-$BMO$ estimate due to Fefferman and Stein \cite{FS}, we have
\[
	[\nabla q_0]_{BMO^\infty(\mathbf{R}^n)}
	\leq C_0[\overline{v}]_{BMO^\infty(\mathbf{R}^n)}
\]
with $C_0>0$ independent of $\overline{v}$.
However, it is difficult to control $[\nabla d_\Gamma \cdot \nabla q_0]_{b^\nu}$ so we construct another function $q_1$ instead of $q_0$.

Although $BMO$ space does not allow the standard cut-off procedure, our space is in $L^1$, so we are able to decompose $v$ into two parts $v=v_1+ v_2$ such that the support of $v_2$ is close to $\Gamma$ while the support of $v_1$ is away from $\Gamma$;
see Proposition \ref{2M}.
For $v_1$ we just set
\[
	q^1_1 = E * \operatorname{div}v_1
\]
by extending $v_1$ as zero outside its support.
Then, the $L^\infty$ bound for $\nabla q^1_1$ is well controlled near $\Gamma$, which yields a bound for $b^\nu$ semi-norm.
To estimate $v_2$, we use a normal coordinate system near $\Gamma$ and reduce the problem to the half space. 
Let $d$ denotes the signed distance function where $d=d_\Gamma$ in $\Omega$ and $d = - d_\Gamma$ outside $\Omega$.
We extend $v_2$ to $\mathbf{R}^n$ so that the normal part $(\nabla d\cdot\overline{v}_2) \nabla d$ is odd and the tangential part $\overline{v_2} - (\nabla d \cdot \overline{v_2}) \nabla d$ is even in the direction of $\nabla d$ with respect to $\Gamma$.
In such type of coordinate system, the minus Laplacian can be transformed as
\[
	L = A - B + \text{lower order terms},\ 
	A = -\Delta_\eta, \ B = \sum_{1\leq i,j\leq n-1} \partial_{\eta_i} b_{ij} \partial_{\eta_j},
\]
where $\eta_n$ is the normal direction to the boundary so that $\{ \eta_n>0\}$ is the half space.
By choosing a suitable coordinate system to represent $\Gamma$ locally, we are able to arrange $b_{ij}=0$ at one point of the boundary of the local coordinate system.
We use a freezing coefficient method to construct volume potential $q_1^2$ and $q^3_1$, which corresponds to the contribution from the tangential part $\overline{v_2}^{\mathrm{tan}}$ and the normal part $\overline{v_2}^{\mathrm{nor}}$ respectively. Since the leading term of $\operatorname{div} \overline{v_2}^{\mathrm{nor}}$ in normal coordinate consists of the differential of $\eta_n$ only,
if we extend the coefficient $b_{ij}$ even in $\eta_n$, 
$q_1^3$ is constructed so that the leading term of $\nabla d \cdot \nabla q^3_1$ is odd in the direction of $\nabla d$.
On the other hand, as the leading term of $\operatorname{div} \overline{v_2}^{\mathrm{tan}}$ in normal coordinate consists of the differential of $\eta' = (\eta_1, ... , \eta_{n-1})$ only, the even extension of $b_{ij}$ in $\eta_n$ gives rise to $q_1^2$ so that the leading term of $\nabla d \cdot \nabla q_1^2$ is also odd in the direction of $\nabla d$.
Disregarding lower order terms and localization procedure, we set $q_1^2$ and $q^3_1$ of the form
\begin{align*}
	q_1^2 &= -L^{-1} \operatorname{div}\overline{v}^{\mathrm{tan}}_2 = -A^{-1} (I-BA^{-1})^{-1} \operatorname{div}\overline{v}^{\mathrm{tan}}_2, \\
	q^3_1 &= -L^{-1} \operatorname{div}\overline{v}^{\mathrm{nor}}_2 = -A^{-1} (I-BA^{-1})^{-1} \operatorname{div}\overline{v}^{\mathrm{nor}}_2.
\end{align*}
One is able to arrange $BA^{-1}$ small by taking a small neighborhood of a boundary point.
Then $(I-BA^{-1})^{-1}$ is given as the Neumann series $\sum^\infty_{m=0}(BA^{-1})^m$.
We are able to establish $BMO$-$BMO$ estimate for $\nabla q_1^2$ and $\nabla q^3_1$, i.e.
\[
	\left[ \nabla q_1^2 \right]_{BMO(\mathbf{R}^n)} \leq C'_0 \left[ \overline{v}^{\mathrm{tan}}_2 \right]_{BMO(\mathbf{R}^n)}, \;
	\left[ \nabla q^3_1 \right]_{BMO(\mathbf{R}^n)} \leq C'_0 \left[ \overline{v}^{\mathrm{nor}}_2 \right]_{BMO(\mathbf{R}^n)}
\]
with some constant $C'_0$ independent of $\overline{v_2}$.
Since the leading term of $\nabla d \cdot (\nabla q_1^2 + \nabla q^3_1)$ is odd in the direction of $\nabla d$ with respect to $\Gamma$, the $BMO$ bound implies $b^\nu$ bound.
Note that $\left[\overline{v_2}^{\mathrm{nor}}\right]_{BMO(\mathbf{R}^n)}$ is controlled by $[v_2]_{b^\nu}$ and $[v_2]_{BMO(\Omega)}$ since $\overline{v_2}^{\mathrm{nor}}$ is odd in the direction of $\nabla d$ with respect to $\Gamma$.
By the procedure sketched above, we are able to construct a suitable operator by setting $q_1=q^1_1+q^2_1+q^3_1$.
\begin{theorem}[Construction of a suitable volume potential] \label{CSV}
Let $\Omega$ be a bounded $C^3$ domain in $\mathbf{R}^n$.
 Then, there exists a linear operator $v\longmapsto q_1$ from $vBMO(\Omega)$ to $L^\infty(\Omega)$ such that
\[
	- \Delta q_1 = \operatorname{div}v \quad\text{in}\quad \Omega
\]
and that there exists a constant $C_1=C_1(\Omega)$ satisfying
\[
	\|\nabla q_1\|_{vBMO(\Omega)} \leq C_1 \|v\|_{vBMO(\Omega)}.
\]
In particular, the operator $v\longmapsto\nabla q_1$ is a bounded linear operator in $vBMO(\Omega)$.
\end{theorem}
%

By this operator, we observe that $w=v-\nabla q_1$ is divergence free in $\Omega$.
 Unfortunately, this $w$ may not fulfill the trace condition $w\cdot\mathbf{n}=0$ on the boundary $\Gamma$.
 We construct another potential $q_2$ by solving the Neumann problem
\begin{align*}
	\Delta q_2 &= 0 \quad\text{in}\quad \Omega \\
	\frac{\partial q_2}{\partial n} &= w\cdot\mathbf{n} \quad\text{on}\quad \Gamma.
\end{align*}
We then set $q=q_1+q_2$.
 Since $\partial q_2/\partial\mathbf{n}=\nabla q_2\cdot \mathbf{n}$, $v_0=v-\nabla q$ gives the Helmholtz decomposition \eqref{H}.
 To complete the proof of Theorem \ref{M}, it suffices to prove that $\|\nabla q_2\|_{vBMO(\Omega)}$ is bounded by a constant multiply of $\|v\|_{vBMO(\Omega)}$.
\begin{lemma}[Estimate of the normal trace] \label{ET}
Let $\Omega$ be a bounded $C^{2+\kappa}$ domain in $\mathbf{R}^n$ with $\kappa \in (0,1)$.
 Then there is a constant $C_2 = C_2(\Omega)$ such that
\[
	\| w \cdot \mathbf{n} \|_{L^\infty(\Gamma)} 
	\leq C_2 \| w \|_{vBMO(\Omega)}
\]
for all $w \in vBMO(\Omega)$ with $\operatorname{div}w = 0$.
\end{lemma}

This is a special case of the trace theorem established in \cite{GigaGu2}.
 We finally need the estimate for the Neumann problem.
\begin{lemma}[Estimate for the Neumann problem] 
\label{EN}
Let $\Omega$ be a bounded $C^2$ domain.
 For $g \in L^\infty(\Gamma)$ satisfying $\int_\Gamma g \, d\mathcal{H}^{n-1}=0$, there exists a unique (up to constant) solution $u$ to the Neumann problem
%
\begin{align}
&\begin{aligned} \label{1NP} 
	\Delta u &= 0 \quad\text{in}\quad \Omega \\
	\frac{\partial u}{\partial \mathbf{n}} &= g \quad\text{on}\quad \Gamma
\end{aligned}
\end{align}
such that the operator $g \longmapsto u$ is linear and that there exists a constant $C_3 = C_3(\Omega)$ such that
\[
	\| \nabla u \|_{vBMO(\Omega)} \leq C_3 \|g\|_{L^\infty(\Gamma)}.
\]
\end{lemma}

Combining these two lemmas, Theorem \ref{CSV} yields
\begin{align*}
	\|\nabla q_2\|_{vBMO(\Omega)} &\leq C_3 C_2 \|v-\nabla q_1\|_{vBMO(\Omega)} \\
	&\leq C_3 C_2 (1 + C_1) \|v\|_{vBMO(\Omega)}.
\end{align*}
Setting $q=q_1+q_2$ and $v_0=v-\nabla q$, we now observe that the projections $v\longmapsto v_0$, $v\longmapsto \nabla q$ are bounded in $vBMO(\Omega)$, which yields \eqref{HD} in Theorem \ref{M}.

To show Lemma \ref{EN} let $N(x,y)$ be the Neumann Green function.
Then a solution of \eqref{1NP} is given by $\int_{\Gamma} N(x,y) g(y) \, \mathrm{d} \mathcal{H}^{n-1}$.
It is well-known (see e.g.\ \cite[Appendix]{GGH}) that leading part of $N$ is $E(x-y)$. We have to estimate
\begin{equation*}
	\left\| \nabla E \ast (\delta_{\Gamma} \otimes g)\right\|_{vBMO^{\infty,\nu}(\Omega)}.
\end{equation*}
Here $\delta_{\Gamma}$ denotes the delta function supported on $\Gamma$, i.e.,
\begin{align*}
\delta_{\Gamma} : \psi \mapsto \int_{\Gamma} \psi \, \mathrm{d} \mathcal{H}^{n-1}
\end{align*}
for $\psi \in C_{c}^{\infty}(\mathbb{R}^n)$. We take a $C^{2}$ cutoff function $\theta \geq 0$ such that $\theta(\sigma) = 1$ for $\sigma \leq 1$, $\theta(\sigma) = 0$ for $\sigma \geq 2$. We take $\delta$ small so that $2\delta$ is smaller than the reach of $\Gamma$. By this choice, $\theta_d = \theta(d/\delta)$ is $C^{2}$ in $\mathbf{R}^n$, where $d$ denotes the signed distance function from $\Gamma$ so that $\nabla d = - \mathbf{n}$ on $\Gamma$. 
For $g \in L^{\infty}(\Gamma)$, we extend $g$ so that $\nabla d \cdot g = 0$ near the $2\delta$-neighborhood of $\Gamma$. Let $g_{e}$ denotes this extension and set $g_{e,c} = \theta_d g_e$. A key observation is that
\begin{align*}
&\delta_{\Gamma} \otimes g = ( \nabla 1_{\Omega} \cdot \nabla d ) g_{e,c} = \operatorname{div} \, (g_{e,c} 1_{\Omega} \nabla d) - 1_{\Omega} \operatorname{div} \, (g_{e,c} \nabla d) \\
&\operatorname{div} \, (g_{e,c} \nabla d) = g_{e,c} \Delta d + \nabla d \cdot \nabla g_{e,c} = g_{e,c} \Delta d + \frac{\theta'(d/\delta)}{\delta} g_e,
\end{align*}
where $1_{\Omega}$ is the characteristic function of $\Omega$. The leading (singular) part of $\nabla E \ast (\delta_{\Gamma} \otimes g)$ is the term involving $\operatorname{div} \, (g_{e,c} 1_{\Omega} \nabla d)$.
 The famous $L^\infty$-$BMO$ estimate for the singular integral operator $\nabla E*\operatorname{div}$ yields
\begin{align*}
	\left\| \nabla E* \operatorname{div}\,(g_{e,c} 1_{\Omega} \nabla d)\right\|_{BMO(\mathbf{R}^n)}
	\leq C \|g_{e,c} \nabla d\|_{L^{\infty}(\Omega)}
	\leq C' \|g\|_{L^{\infty}(\Gamma)}
\end{align*}
with $C$ and $C'$ independent of $g$.
 All other terms can be estimated easily since the integral kernel is integrable.
 A direct calculation gives an $L^\infty$ estimate near $\Gamma$ for $\nabla d\cdot\nabla E*(\delta_\Gamma\otimes g)$ which yields
\[
	\left[ \nabla d \cdot \nabla E \ast (\delta_\Gamma \otimes g) \right]_{b^\nu}
	\leq C_4 \| g \|_{L^\infty(\Gamma)}
\]
with $C_4$ independent of $g$, but it is impossible to estimate $b^\nu$-seminorm of the tangential part.
This is the main reason why we use $vBMO$ instead of $BMO_b$-type space where $b^\nu$-boundedness of ALL components of vector fields is imposed; see the end of Section \ref{END}.

To extend our results to a more general domain it seems to be reasonable to consider $vBMO\cap L^2$. 
This is because $L^p\cap L^2\ (p>2)$ admits the Helmholtz decomposition for arbitrary uniformly $C^2$ domains as proved in \cite{FKS1}, \cite{FKS2}.

Our approach in this paper is to derive the boundedness of the operator $v\mapsto\nabla q$ by a potential-theoretic approach.
In $L^p$ setting there is a variational approach based on duality introduced by \cite{SS}; see also \cite{FKS1}. 
The key estimate is 
\[
	\|\nabla q\|_{L^p(\Omega)} \leq C_5 \sup
	\left\{ \int_\Omega \nabla q\cdot\nabla\varphi \, dx \Bigm|
	\|\nabla\varphi\|_{L^{p'}(\Omega)} \leq 1 \right\}
\]
with $C_5$ independent of $q$, where $1/p+1/p'=1$, $1<p<\infty$.
 Formally, this estimate yields the desired bound $\|\nabla q\|_{L^p(\Omega)}\leq C_5\|v\|_{L^p(\Omega)}$ since $\int_\Omega \nabla q\cdot\nabla\varphi \, dx = \int_\Omega v\cdot\nabla\varphi \, dx$.
 At this moment, it is not clear that similar estimate holds if one replaces $L^p(\Omega)$ by $vBMO$ since the predual space of $vBMO$ is not clear.

For $BMO_b$ type solution, it is known that the Stokes semigroup is analytic \cite{BG}, \cite{BGS}.
 However, it is nontrivial to extend to the space $vBMO$ since in the half space the Stokes operator with Dirichlet boundary condition does not generate a semigroup because $\left[u(t)\right]_{vBMO}$ for the solution $u(t)$ may be non-zero for $t>0$ for initial data $u_0$ with $[u_0]_{vBMO}=0$ so that $u^{\tan}_0$ may be a non-zero constant \cite[Example 6.5]{BG}.

This paper is organized as follows.
In Section \ref{sec:1}, to construct a volume potential of $\operatorname{div}v$, we localize the problem and reduce the problem to small neighborhoods of points on the boundary.
In Section \ref{sec:2}, we construct a leading part of the volume potential by a perturbation method called the freezing coefficient method.
In these two sections, we complete the proof of Theorem \ref{CSV}.
In Section \ref{NPB}, we prove Lemma \ref{EN} by estimating the single layer potential.

\section{Construction of volume potentials} 
\label{sec:1}

For $v \in vBMO(\Omega)$, we shall construct a suitable potential $q_1$ so that $v\longmapsto\nabla q_1$ is a bounded linear operator in $vBMO$ as stated in Theorem \ref{CSV}.
 In this section, as a preliminary, we reduce the problem to the case that the support of $v$ is contained in a small neighborhood of a point of the boundary and it consists of only normal part.
 
\subsection{Localization procedure} 
\label{LOC}
Let $\Omega$ be a uniformly $C^k$ domain in $\mathbf{R}^n$ ($k\geq 1$).
In other words, there exists $r_*, \delta_*>0$ such that for each $z_0 \in \Gamma$, up to translation and rotation, there exists a function $h_{z_0}$ which is $C^k$ in a closed ball $B_{r_*}(0')$ of radius $r_*$ centered at the origin $0'$ of $\mathbf{R}^{n-1}$ satisfying following properties:
\begin{enumerate}
\item[(i)] $K_\Gamma := \sup_{B_{r_*}(0')} \left| (\nabla')^s h_{z_0} \right| < \infty$ for $s=0,1,2,\ldots,k$, where $\nabla'$ denotes the gradient in $x'\in\mathbf{R}^{n-1}$;
 $\nabla'h(0')=0$, $h(0')=0$,
\item[(i\hspace{-1pt}i)] $\Omega\cap U_{r_*,\delta_*, h_{z_0}}(z_0)=\left\{ (x',x_n)\in\mathbf{R}^n \bigm| h_{z_0}(x')<x_n<h_{z_0}(x')+\delta_*,\ |x'|<r_* \right\}$ for
\[
	U_{r_*,\delta_*, h_{z_0}}(z_0) := \left\{ (x',x_n)\in\mathbf{R}^n \bigm| h_{z_0}(x')-\delta_* < x_n < h_{z_0}(x')+\delta_*,\ |x'|<r_* \right\},
\]
\item[(i\hspace{-1pt}i\hspace{-1pt}i)] $\Gamma\cap U_{r_*,\delta_*, h_{z_0}}(z_0)=\left\{ (x',x_n)\in\mathbf{R}^n \bigm| x_n = h_{z_0}(x'),\ |x'|<r_* \right\}$.
\end{enumerate}
A bounded $C^k$ domain is, of course, a uniformly $C^k$ domain.

Let $d$ denote the signed distance function from $\Gamma$ which is defined by
\begin{equation} 
	d(x) = \left \{
\begin{array}{r}
	\displaystyle \inf_{y\in\Gamma}|x-y| \quad\text{for}\quad x\in\Omega, \\
	\displaystyle -\inf_{y\in\Gamma}|x-y| \quad\text{for}\quad x\notin\Omega 
\end{array}
	\right.
\end{equation}
%
so that $d(x)=d_\Gamma(x)$ for $x\in\Omega$.
 If $\Omega$ is a bounded $C^2$ domain, then there is $R_*>0$ such that if $\left|d(x)\right|<R_*$, there is unique point $\pi x$ such that $|x-\pi x|=\left|d(x)\right|$.
 The supremum of such $R_*$ is called the reach of $\Omega$ and $\Omega^c$.
 Moreover, $d$ is $C^2$ in the $R_*$-neighborhood of $\Gamma$, i.e., $d \in C^2\left(\Gamma^{\mathbf{R}^n}_{R_*}\right)$ with
\[
	\Gamma^{\mathbf{R}^n}_{R_\ast} := \left\{ x \in \mathbf{R}^n \bigm|
	\left| d(x) \right| < R_\ast \right\};
\]
see \cite[Chap.\ 14, Appendix]{GT}, \cite[\S 4.4]{KP}.
Note that $R_*$ satisfies
\[
	R_* = \min \left( R^\Omega_*, R^{\Omega^c}_* \right),
\]
where $R^\Omega_*$ is the reach of $\Gamma$ in $\Omega$ while $R^{\Omega^c}_*$ is the reach of $\Gamma$ in the complement $\Omega^c$ of $\Omega$.
Let $K_\Gamma^\ast := \mathrm{max} \, \{ K_\Gamma, 1 \}$.
There exists $0 < \rho_0 < \min(r_*, \delta_*, \frac{R_*}{2}, \frac{1}{2n K_\Gamma^\ast})$ such that
\[
	U_\rho(z_0) := \left\{ x \in \mathbf{R}^n \bigm|
	(\pi x)' \in \operatorname{int} B_\rho(0'),\ 
	\left| d(x) \right| < \rho \right\}
\]
is contained in the coordinate chart $U_{r_*,\delta_*, h_{z_0}}(z_0)$ for any $\rho \leq\rho_0$.

We always take $\rho<\rho_0$.
Since $\Omega$ is bounded and
\[
	\bigcup_{z \in \Gamma} U_\rho(z)
\]
covers the compact set $K=\mathrm{cl}\left(\Gamma^{\mathbf{R}^n}_{\rho/2}\right)$, there exists a finite subcover $\left\{ U_\rho(z_j) \right\}^m_{j=1}$ of $K$, where the number $m$ depends on $\rho$.
For $\sigma > 0$, we denote that
\[
	\Omega^\sigma = \Omega \backslash \Gamma^{\mathbf{R}^n}_\sigma, \; \, U_{\sigma,j} := U_\sigma(z_j).
\]
Observe that
\[
	\overline{\Omega} \subset \bigcup^m_{j=1} U_{\rho,j} \cup \Omega^{\rho/2}.
\]
Let $\{ \varphi_j \}^m_{j=0}$ be a partition of the unity associated with $\{ U_{\rho,j}\} \cup \{\Omega^{\rho/2}\}$ in the sense that
\begin{align*}
	&\varphi_j \in C^\infty_c (U_{\rho,j} \cap \overline{\Omega}), \; \; 0 \leq \varphi_j \leq 1 \quad \text{for} \quad j=1, \ldots, m, \\
	&\varphi_0 \in C^\infty_c (\Omega^{\rho/2}), \; \; 0 \leq \varphi_0 \leq 1, \; \; \varphi_0 = 1 \quad \text{in} \quad \Omega^{\rho}
\end{align*}
and
\[
	\sum^m_{j=0} \varphi_j = 1 \quad \text{in} \quad \overline{\Omega}.
\]
Here $C^\infty_c(W)$ denotes the space of all smooth function in $W$ whose support is compact in $W$. 

Throughout this paper, unless otherwise specified, the symbol $C$ in an inequality represents a positive constant independent of quantities that appeared in the inequality. For a fixed $\rho>0$, $C_\rho$ represents a constant depending only on $\rho$. $C_n$ represents a constant depending only on $n$ and $C_{\Omega,n}$ represents a constant depending only on $\Omega$ and $n$.

\subsection{Cut-off and extension} 
\label{CUT}

In general, multiplication by a smooth function to $BMO$ is not bounded in $BMO$.
 Fortunately, our space is closed by multiplication.
%
\begin{proposition}[Multiplication] \label{2M}
Let $\Omega$ be a bounded $C^2$ domain in $\mathbf{R}^n$.
 Let $\varphi \in C^\gamma(\Omega)$, $\gamma\in(0,1)$.
 For each $v \in vBMO(\Omega)$, the function $\varphi v \in vBMO(\Omega)$ satisfies
\[
	\| \varphi v \|_{vBMO(\Omega)}
	\leq C\|\varphi\|_{C^\gamma(\Omega)}
	\|v\|_{vBMO(\Omega)}
\]
with $C$ independent of $\varphi$ and $v$.
\end{proposition}
\begin{proof}
Since
\[
	\left[ \nabla d\cdot\varphi v \right]_{b^\nu}
	\leq \|\varphi\|_{L^\infty(\Omega)} \left[ \nabla d\cdot v \right]_{b^\nu},
\]
it suffices to establish the estimate
\begin{equation} \label{2MI}
	\left[ \varphi v \right]_{BMO(\Omega)}
	\leq c_0 \|\varphi\|_{C^\gamma(\Omega)} \|v\|_{vBMO(\Omega)}
\end{equation}
with $c_0$ independent of $\varphi$ and $v$.
Since a bounded Lipschitz domain is a uniform domain, we are able to apply \cite[Theorem 13]{GigaGu2} to get
\[
	\left[ \varphi v \right]_{BMO(\Omega)}
	\leq c_1 \|\varphi\|_{C^\gamma(\Omega)}
	( [v]_{BMO(\Omega)} + \|v\|_{L^1(\Omega)} ).
\]
This is based on the product estimate of a H\"older function and a function in $bmo(\mathbf{R}^n) := BMO(\mathbf{R}^n)\cap L^1_\mathrm{ul}(\mathbf{R}^n)$ where
\[
	L^1_{\mathrm{ul}}(\mathbf{R}^n) := 
	\bigg\{ f \in L^1_{\mathrm{loc}} (\mathbf{R}^n) \biggm| 
	\| f \|_{L^1_{\mathrm{ul}}(\mathbf{R}^n)} := \sup_{x \in \mathbf{R}^n} \int_{B_1(x)} \bigl| f(y) \bigr| \, dy < \infty \bigg\}.
\]
The space $bmo(\mathbf{R}^n)$ is equipped with the norm
\[
\| f \|_{bmo(\mathbf{R}^n)}:= [f]_{BMO(\mathbf{R}^n)} + \| f \|_{L_{\mathrm{ul}}^1(\mathbf{R}^n)}
\]
for $f \in bmo(\mathbf{R}^n)$.
The product estimate for $bmo$ follows from a similar result for a local Hardy space $h^1=F^0_{1,2}$ \cite[Remark 4.4]{Sa} and duality $bmo=(h^1)'$ \cite[Theorem 3.26]{Sa}.
To handle a function in $\Omega$, we need an extension to conclude \cite[Theorem 13]{GigaGu2}.
Fortunately, by the characterization of $vBMO$ for a bounded $C^2$ domain \cite[Theorem 9]{GigaGu2},
\[
	\|v\|_{L^1(\Omega)} \leq c_2 \|v\|_{vBMO(\Omega)}.
\]
Here $c_j$ denotes a constant independent of $v$ and $\varphi$ for $j=1,2$.
Combining these two estimates, we obtain \eqref{2MI} with $c_0 = c_1(1+c_2)$.
This yields Proposition \ref{2M}. 
\end{proof}
%

For a bounded $C^3$ domain, we next consider an extension based on the normal coordinate in $U_\rho(z_0)$ for $\rho\leq\rho_0$ of the form
\begin{eqnarray} \label{NCC}
\left\{
\begin{array}{lcl}
x' &=& \eta' + \eta_n \nabla'd ( \eta', h_{z_0}(\eta') ); \\
x_n &=& h_{z_0}(\eta') + \eta_n \partial_{x_n} d ( \eta', h_{z_0}(\eta') ).
\end{array}
\right.
\end{eqnarray}
Let $V_{\sigma} := B_\sigma(0') \times (- \sigma, \sigma)$ for $\sigma \in (0,\rho_0)$. We shall write this coordinate change by $x=\psi(\eta)$ with $\psi \in C^2(V_{\rho_0})$ and
\[
	x = \pi x - d(x) \mathbf{n}(\pi x), \quad
	\mathbf{n}(\pi x) = -\nabla d(\pi x).
\]
We consider the projection to the direction to $\nabla d$. For $x \in \Gamma_{\rho_0}^{\mathbf{R}^n}$, we set
\[
	P(x) = \nabla d(\pi x) \otimes \nabla d(\pi x) = \mathbf{n}(\pi x) \otimes \mathbf{n}(\pi x).
\]
For later convenience, we set $Q(x) = I - P(x)$ which is the tangential projection for $x \in \Gamma_{\rho_0}^{\mathbf{R}^n}$.
For a function $f$ in $\Gamma_\rho^{\mathbf{R}^n} \cap \overline{\Omega}$, let $f_{\mathrm{even}}$ (resp.\ $f_{\mathrm{odd}}$) denote its even (odd) extension to $\Gamma_\rho^{\mathbf{R}^n}$ defined by
\begin{align*}
	f_{\mathrm{even}} \left( \pi x + d(x)\mathbf{n}(\pi x) \right)
	&= f \left( \pi x - d(x)\mathbf{n}(\pi x) \right)
	&\text{for}\quad x \in \Gamma_\rho^{\mathbf{R}^n} \backslash \overline{\Omega}, \\
	f_{\mathrm{odd}} \left( \pi x + d(x)\mathbf{n}(\pi x) \right)
	&= -f \left( \pi x - d(x)\mathbf{n}(\pi x) \right)
	&\text{for}\quad x \in \Gamma_\rho^{\mathbf{R}^n} \backslash \overline{\Omega}.
\end{align*}
We denote $r_W$ to be the restriction in $W$ for any subset $W \subset \mathbf{R}^n$. Let $f$ be a function (or a vector field) defined in $V_\sigma$ for some $\sigma \in (0,\infty]$. We set $E_{\mathrm{even}} f$ to be the even extension of $f$ in $V_\sigma \cap \mathbf{R}_+^n$ to $V_\sigma$ with respect to the $n$-th variable, i.e.,
\[
E_{\mathrm{even}} f (\eta', - \eta_n) = f(\eta', \eta_n)
\]
for any $(\eta', \eta_n) \in V_\sigma \cap \mathbf{R}_+^n$.

For $v\in vBMO(\Omega)$ with $\operatorname{supp} \, v \subset U_\rho(z_0) \cap \overline{\Omega}$, let $\overline{v}$ be its extension of the form
\begin{align} \label{SE}
	\overline{v}(x) := (P v_{\mathrm{odd}}) (x) + (Q v_{\mathrm{even}}) (x)
\end{align}
for $x \in U_{\rho}(z_0)$. Notice that $\operatorname{supp} \, \overline{v} \subset U_{\rho}(z_0)$, $\overline{v}$ is indeed defined in $\mathbf{R}^n$ with $\overline{v}(x) = 0$ for any $x \in U_\rho(z_0)^{\mathrm{c}}$. Define 
\[
L_\ast := \underset{z_0 \in \Gamma, \, \rho \leq \rho_0}{\sup} \, \mathrm{max} \, \{ \| \nabla \psi \|_{L^\infty(V_\rho)} + \| \nabla \psi^{-1} \|_{L^\infty(U_{\rho}(z_0))}, 1 \}. 
\]
Since the boundary $\Gamma$ is uniformly $C^3$, $L_\ast$ is finite that depends on $\Omega$ only. We set $\rho_{0,\ast} = \rho_0/12L_\ast$.
\begin{proposition} \label{2E}
Let $\Omega \subset \mathbf{R}^n$ be a bounded $C^2$ domain, $z_0 \in \Gamma$ and $\rho \in (0, \rho_{0,\ast})$.
There exists a constant $C_\rho$, which depends on $\rho$ only, such that
\begin{align*}
	[\overline{v}]_{BMO\left( \mathbf{R}^n \right)} &\leq C_\rho \|v\|_{vBMO(\Omega)}, \\
	[\nabla d\cdot\overline{v}]_{b^\nu(\Gamma)} &\leq C_\rho \|v\|_{vBMO(\Omega)}
\end{align*}
for all $v \in vBMO(\Omega)$ with $\operatorname{supp} \, v \subset U_\rho(z_0) \cap \overline{\Omega}$ and $\nu>0$.
\end{proposition}
In the normal coordinate, $P\overline{v} = P v_{\mathrm{odd}}$ is odd in $\eta_n$ and $Q\overline{v} = Q v_{\mathrm{even}}$ is even in $\eta_n$. The key idea of proving this proposition is to reduce the problem to the case where the boundary is locally flat by invoking the normal coordinate.
%
\begin{proof}
Since $vBMO(\Omega) \subset L^1(\Omega)$, see e.g. \cite[Theorem 9]{GigaGu2}, by considering the normal coordinate change $y = \psi(\eta)$ in $U_{\rho}(z_0)$ we can deduce that $v_{\mathrm{even}}, v_{\mathrm{odd}} \in L^1(\mathbf{R}^n)$ satisfying
\[
\| v_{\mathrm{even}} \|_{L^1(\mathbf{R}^n)} = \| v_{\mathrm{odd}} \|_{L^1(\mathbf{R}^n)} \leq 2 L_\ast^2 \| v \|_{L^1(\Omega)}.
\]
Hence $\overline{v} \in L^1(\mathbf{R}^n)$ satisfies the estimate $\| \overline{v} \|_{L^1(\mathbf{R}^n)} \leq C_{\Omega,n} \| v \|_{L^1(\Omega)}$. Since $\Omega$ is a uniform domain, by \cite[Theorem 1]{PJ} there exists $v_J \in BMO(\mathbf{R}^n)$ with $r_\Omega v_J = v$ and 
\[
[v_J]_{BMO(\mathbf{R}^n)} \leq C_{\Omega,n} [v]_{BMO^\infty(\Omega)}.
\]

Suppose that $B_r(\zeta) \subset V_{4 \rho L_\ast}^+ := V_{4 \rho L_\ast} \cap \mathbf{R}_+^n$. The mean value theorem implies that $\psi(B_r(\zeta)) \subset B_{L_\ast r}(x)$ with $x = \psi(\zeta)$. By change of variables $y = \psi(\eta)$ in $U_{4 \rho L_\ast}(z_0)$, we see that
\begin{align*}
\frac{1}{|B_r(\zeta)|} \int_{B_r(\zeta)} | v \circ \psi (\eta) - c | \, d \eta &\leq L_\ast \cdot \frac{1}{|B_r(\zeta)|} \int_{\psi(B_r(\zeta))} | v(y) - c | \, dy \\
&\leq C_n L_\ast^{n+1} \cdot \frac{1}{|B_{L_\ast r}(x)|} \int_{B_{L_\ast r}(x)} | v_J (y) -c | \, dy
\end{align*}
for any constant vector $c \in \mathbf{R}^n$. By considering an equivalent definition of the $BMO$-seminorm, see e.g. \cite[Proposition 3.1.2]{Gra}, we deduce that
\[
[v \circ \psi]_{BMO^\infty(V_{4 \rho L_\ast}^+)} \leq C_{\Omega,n} [v]_{BMO^\infty(\Omega)}.
\]
By recalling the results concerning the even extension of $BMO$ functions in the half space, see \cite[Lemma 3.2]{GigaGu} and \cite[Lemma 3.4]{GigaGu}, we can deduce that
\begin{align} \label{EBE}
[ v_{\mathrm{even}} \circ \psi ]_{BMO^\infty(V_{4 \rho L_\ast})} \leq C_{\Omega,n} [v]_{BMO^\infty(\Omega)}.
\end{align}

Next, we shall estimate the $BMO$-seminorm of $v_{\mathrm{even}}$. Let $B_r(x)$ be a ball with radius $r \leq \frac{\rho}{2}$. If either $B_r(x) \cap U_\rho(z_0) = \emptyset$ or $B_r(x) \subset \Omega$, there is nothing to prove. It is sufficient to consider $B_r(x)$ that intersects both $U_\rho(z_0)$ and $\Omega^{\mathrm{c}}$. In this case we can find $x_0 \in B_r(x) \cap U_\rho(z_0)$. Since $B_r(x) \subset B_{2r}(x_0) \subset B_{4 \rho}(z_0) \subset U_{8 \rho}(z_0)$, by considering change of variables $y = \psi(\eta)$ in $U_{8 \rho}(x_0)$, we have that
\[
\frac{1}{|B_r(x)|} \int_{B_r(x)} | v_{\mathrm{even}} (y) - c | \, dy \leq \frac{L_\ast}{|B_r(x)|} \int_{\psi^{-1}(B_r(x))} | v_{\mathrm{even}} \circ \psi (\eta) - c | \, d\eta.
\]
For any $y \in B_r(x)$, we have that $| y - z_0 | < 4 \rho$. Hence $\psi^{-1}(B_r(x)) \subset B_{L_\ast r}(\zeta) \subset B_{4 \rho L_\ast}(0) \subset V_{4 \rho L_\ast}$. By (\ref{EBE}), we deduce that
\[
\frac{1}{|B_r(x)|} \int_{B_r(x)} | v_{\mathrm{even}} (y) - (v_{\mathrm{even}})_{B_r(x)} | \, dy \leq C_{\Omega,n} [v]_{BMO^\infty(\Omega)}.
\]
Thus, we obtain that
\[
[v_{\mathrm{even}}]_{BMO^{\frac{\rho}{2}}(\mathbf{R}^n)} \leq C_{\Omega,n} [v]_{BMO^\infty(\Omega)}.
\]
For a ball $B$ with radius $r(B) > \frac{\rho}{2}$, a simple triangle inequality implies that
\[
\frac{1}{|B|} \int_B | v_{\mathrm{even}} (y) - (v_{\mathrm{even}})_B | \, dy \leq \frac{2}{|B|} \int_B | v_{\mathrm{even}} (y) | \, dy \leq \frac{C_n}{\rho^n} \| v_{\mathrm{even}} \|_{L^1(\mathbf{R}^n)}.
\]
Therefore, we obtain the $BMO$ estimate for $v_{\mathrm{even}}$, i.e.,
\[
[v_{\mathrm{even}}]_{BMO(\mathbf{R}^n)} \leq \frac{C_{\Omega,n}}{\rho^n} \| v \|_{vBMO(\Omega)}.
\]

We shall then give the $BMO$ estimate for $P v_{\mathrm{odd}}$. Since $\nabla d \in C^1(\Gamma_{\rho_0}^{\mathbf{R}^n})$, there exists $D_e \in C^1(\mathbf{R}^n)$ such that $\| D_e \|_{C^1(\mathbf{R}^n)} \leq \| \nabla d \|_{C^1(\Gamma_{\rho_0}^{\mathbf{R}^n})}$ and $r_{\Gamma_{\rho_0}^{\mathbf{R}^n}} D_e = \nabla d$, see the proof of \cite[Theorem 13]{GigaGu2}. By the multiplication rule for $bmo$ functions, we have that $(P v)_E := (D_e \cdot v_{\mathrm{even}} ) D_e \in bmo(\mathbf{R}^n)$, see also \cite[Theorem 13]{GigaGu2}. Consider the normal coordinate change in $U_{4 \rho L_\ast}(z_0)$. Since $(P v)_E = Pv$ in $U_{4 \rho L_\ast}(z_0) \cap \Omega$, same argument in the second paragraph implies that
\[
[P v \circ \psi]_{BMO^\infty(V_{4 \rho L_\ast}^+)} \leq C_{\Omega,n} \| (P v)_E \|_{bmo(\mathbb{R}^n)} \leq \frac{C_{\Omega,n}}{\rho^n} \| v \|_{vBMO(\Omega)}.
\]
Let $\zeta \in V_{12 \rho L_\ast} = \psi^{-1}(U_{12 \rho L_\ast}(z_0))$ with $\zeta_n = 0$. Let $B_r(\zeta) \subset V_{12 \rho L_\ast}$ and $x = \psi(\zeta)$. Since $F(B_r(\zeta) \cap V_{12 \rho L_\ast}^+) \subset B_{L_\ast r}(x) \cap \Omega$, by considering change of variables $y = \psi(\eta)$ in $U_{12 \rho L_\ast}(z_0)$, we can deduce that
\begin{align} \label{LBNE}
\frac{1}{|B_r(\zeta)|} \int_{B_r(\zeta) \cap V_{12 \rho L_\ast}^+} | P v_{\mathrm{odd}} \circ \psi (\eta) | \, d\eta \leq L_\ast^{n+1} [\nabla d \cdot v]_{b^\nu}.
\end{align}
Recall the results concerning the odd extension of $BMO$ functions in the half space, see \cite[Lemma 3.1]{GigaGu}, we have the estimate
\begin{align} \label{OBE}
[P v_{\mathrm{odd}} \circ \psi ]_{BMO^\infty(V_{4 \rho L_\ast})} \leq \frac{C_{\Omega,n}}{\rho^n}  \| v \|_{vBMO(\Omega)}.
\end{align}
By considering (\ref{OBE}) and the fact that $P v_{\mathrm{odd}} = (P v)_E$ in $\Omega$, same argument in the third paragraph implies the $BMO$ estimate for $P v_{\mathrm{odd}}$, i.e.,
\[
[P v_{\mathrm{odd}}]_{BMO(\mathbf{R}^n)} \leq \frac{C_{\Omega,n}}{\rho^n} \| v \|_{vBMO(\Omega)}.
\]
 
Combining the $BMO$ estimates for $v_{\mathrm{even}}$ and $P v_{\mathrm{odd}}$, we have that
\[
[\overline{v}]_{BMO(\mathbf{R}^n)} \leq \frac{C_{\Omega,n}}{\rho^n} \| v \|_{vBMO(\Omega)}.
\]
Notice that $\nabla d \cdot \overline{v} = v_{\mathrm{odd}} \cdot \nabla d$ in $\mathbf{R}^n$. Let $x \in \Gamma$ and $r \leq \frac{\rho}{L_\ast}$. If $B_r(x) \cap U_\rho(z_0) = \emptyset$, then $v_{\mathrm{odd}} = 0$ in $B_r(x)$. Suppose that $B_r(x) \cap U_\rho(z_0) \neq \emptyset$. Then we can find $x_0 \in B_r(x) \cap U_\rho(z_0) \cap \Gamma$. Let $\zeta_0 = \psi^{-1}(x_0)$, we have that $\psi^{-1}(B_r(x)) \subset B_{2 L_\ast r}(\zeta_0) \subset V_{12 \rho L_\ast}$. Hence, 
\begin{align*}
r^{-n} \int_{B_r(x)} | v_{\mathrm{odd}} \cdot \nabla d | \, dy &\leq \frac{2 L_\ast}{r^n} \int_{B_{2 L_\ast r}(\zeta_0) \cap V_{12 \rho L_\ast}^+} | ( v \cdot \nabla d ) \circ \psi | \, d\eta \\
&\leq \frac{2 L_\ast^2}{r^n} \int_{B_{2 L_\ast^2 r}(x_0) \cap \Omega} | \nabla d \cdot v | \, dy \leq C_{\Omega,n} [\nabla d \cdot v]_{b^\nu}.
\end{align*}
For $r > \frac{\rho}{L_\ast}$, we simply have that
\[
r^{-n} \int_{B_r(x)} | v_{\mathrm{odd}} \cdot \nabla d | \, dy \leq \frac{C_{\Omega,n}}{\rho^n} \| v_{\mathrm{odd}} \|_{L^1(\mathbf{R}^n)} \leq \frac{C_{\Omega,n}}{\rho^n} \| v \|_{vBMO(\Omega)}.
\]
\end{proof}

\subsection{Volume potentials} 
\label{VOL}

To construct mapping $v\mapsto q_1$ in Theorem \ref{CSV}, for some $\rho_\ast$ to be determined later in the next section, we localize $v$ by using the partition of the unity $\{\varphi_j\}^m_{j=0}$ associated with the covering
\[
	\{ U_{\rho,j} \}^m_{j=1} \cup \Omega^{\rho/2}
\]
as in Section \ref{LOC}, where $\rho$ is always assumed to satisfy $\rho \leq \rho_\ast/2$.
Here and hereafter we always assumed that $\Omega$ is a bounded $C^3$ domain in $\mathbf{R}^n$.
\begin{proposition} \label{2V}
There exists a constant $C_\rho$, which depends on $\rho$ only, such that
\begin{align*}
	[ \nabla q^1_1 ]_{BMO^\infty(\mathbf{R}^n)}
	& \leq C_\rho \|v\|_{vBMO(\Omega)}, \\
	\| \nabla q^1_1(x) \|_{L^\infty(\Gamma_{\rho/4}^{\mathbf{R}^n})}
	& \leq C_\rho \|v\|_{vBMO(\Omega)}
\end{align*}
for $q^1_1= E*\operatorname{div} \, (\varphi_0 v)$ and $v \in vBMO(\Omega)$.
In particular, 
\[
	\left[ \nabla q^1_1 \right]_{b^\nu(\Gamma)}
	\leq C_\rho \| v \|_{vBMO(\Omega)}
\]
for $\nu < \rho/4$.
\end{proposition}
%
%
\begin{proof}
By the $BMO$-$BMO$ estimate \cite{FS}, we have the estimate
\[
	\left[ \nabla q^1_1 \right]_{BMO(\mathbf{R}^n)}
	\leq C [\varphi_0 v]_{BMO(\mathbf{R}^n)}.
\]
Consider $x \in \Gamma_{\rho/4}^{\mathrm{R}^n}$. Since $\nabla q_1^1$ is harmonic in $\Gamma_{\rho/2}^{\mathrm{R}^n}$ and $B_{\frac{\rho}{4}}(x) \subset \Gamma_{\rho/2}^{\mathrm{R}^n}$, the mean value property for harmonic functions implies that
\[
\nabla q_1^1 (x) = \frac{C_n}{\rho^n} \int_{B_{\frac{\rho}{4}}(x)} \nabla q_1^1 (y) \, dy.
\]
By H$\ddot{\text{o}}$lder's inequality, we can estimate $| \nabla q_1^1 (x) |$ by $\frac{C_n}{\rho^{n/2}} \| \nabla q_1^1 \|_{L^2(\mathbf{R}^n)}$. Since the convolution with $\nabla^2 E$ is bounded in $L^p$ for any $1<p<\infty$, see e.g. \cite[Theorem 5.2.7 and Theorem 5.2.10]{Gra}, an interpolation inequality (cf. \cite[Lemma 5]{BGST}) implies that
\begin{align*}
\| \nabla q_1^1 \|_{L^2(\mathbf{R}^n)} \leq C \| \varphi_0 v \|_{L^2(\mathbf{R}^n)} \leq C \| \varphi_0 v \|_{L^1(\mathbf{R}^n)}^{\frac{1}{2}} [ \varphi_0 v ]_{BMO(\mathbf{R}^n)}^{\frac{1}{2}}.
\end{align*}
View $\varphi_0 v$ as the extension of $\varphi_0 v$ from $\Omega$ to $\mathbf{R}^n$. By the extension theorem for $bmo$ functions \cite[Theorem 12]{GigaGu2}, we estimate $[ \varphi_0 v ]_{BMO(\mathbf{R}^n)}$ by $C_\rho [ \varphi_0 v ]_{BMO^\infty(\Omega)}$. Since $vBMO(\Omega) \subset L^1(\Omega)$, see \cite[Theorem 9]{GigaGu2}, Proposition \ref{2M} implies that
\[
	| \nabla q^1_1 (x) |
	\leq C_\rho \|v\|_{vBMO(\Omega)}
\]
for any $x \in \Gamma_{\rho/4}^{\mathrm{R}^n}$.
\end{proof}

We next set $v_1 := \varphi_0 v$ and $v_2 := 1-v_1$.
For each $\varphi_j v_2$ ($j=1,.,m$), we extend as in Proposition \ref{2E} to get $\overline{\varphi_j v_2}$ and set
\[
	\overline{v_2} := \sum^m_{j=1} \overline{\varphi_j v_2}.
\]
Indeed, this extension is independent of the choice $\varphi_j$'s but we do not use this fact.
 We next set
\[
	\overline{v_2}^{\tan} := Q \, \overline{v_2} = \sum_{j=1}^m Q \, (\varphi_j v_2)_{\mathrm{even}}.
\]
For $1 \leq j \leq m$, $\varphi_j \in C_{\mathrm{c}}^\infty(U_{\rho,j} \cap \overline{\Omega})$ implies that the even extension of $\varphi_j$ in $U_{\rho,j}$ with respect to $\Gamma$ is H$\ddot{\text{o}}$lder continuous in the sense that $(\varphi_j)_{\mathrm{even}} \in C^{0,1}(U_{\rho,j})$. Moreover, we have that $(\varphi_j)_{\mathrm{even}} \in C^{0,1}(\mathbf{R}^n)$ satisifes 
\[
\| (\varphi_j)_{\mathrm{even}} \|_{C^{0,1}(\mathbf{R}^n)} \leq C_\rho \| (\varphi_j)_{\mathrm{even}} \|_{C^{0,1}(U_{\rho,j})}.
\]

For simplicity of notations, we denote $Q \, (\varphi_j v_2)_{\mathrm{even}}$ by $w_j^{\mathrm{tan}}$ for every $1 \leq j \leq m$. Now, we are ready to construct the suitable potential corresponding to $\overline{v_2}^{\mathrm{tan}}$.
\begin{proposition} \label{VPT}
There exists $\rho_\ast > 0$ such that if $\rho<\rho_\ast/2$, then for every $1 \leq j \leq m$, there exists a linear operator $v \longmapsto p_j^{\mathrm{tan}}$ from $vBMO(\Omega)$ to $L^\infty(\mathbf{R}^n)$ such that
\[
- \Delta p_j^{\mathrm{tan}} = \operatorname{div} w_j^{\mathrm{tan}} \; \; \text{in} \; \; U_{2 \rho,j} \cap \Omega
\]
and that there exists a constant $C_\rho$, independent of $v$, such that
\begin{align*}
	[\nabla p_j^{\mathrm{tan}}]_{BMO(\mathbf{R}^n)} &\leq C_{\rho} \|v\|_{vBMO(\Omega)}, \\
	\sup_{x\in\Gamma, r<\rho} \, \frac{1}{r^n} \int_{B_r(x)} |\nabla d \cdot \nabla p_j^{\mathrm{tan}}| \, dy &\leq C_{\rho} \|v\|_{vBMO(\Omega)}.
\end{align*}
\end{proposition}
%

Having the estimate for the volume potential near the boundary regarding its tangential component, we are left to handle the contribution from $\overline{v}^{\mathrm{nor}}_2:=\overline{v}_2-\overline{v}^{\tan}_2$.
We recall its decomposition
\[
	\overline{v}^{\mathrm{nor}}_2 = \sum^m_{j=1} P \, ( \varphi_j v_2)_{\mathrm{odd}}.
\]
For simplicity of notations, we denote $P \, (\varphi_j v_2)_{\mathrm{odd}}$ by $w_j^{\mathrm{nor}}$ for every $1 \leq j \leq m$.
With a similar idea of proof, we can establish the suitable potential corresponding to $\overline{v}^{\mathrm{nor}}_2$.
\begin{proposition} \label{2LP}
There exists $\rho_\ast > 0$ such that if $\rho<\rho_\ast/2$, then for every $1 \leq j \leq m$, there exists a linear operator $v\longmapsto p_j^{\mathrm{nor}}$ from $vBMO(\Omega)$ to $L^\infty(\mathbf{R}^n)$ such that
\[
- \Delta p_j^{\mathrm{nor}} = \operatorname{div} w_j^{\mathrm{nor}} \; \; \text{in} \; \; U_{2 \rho,j} \cap \Omega
\]
and that there exists a constant $C_{\rho}$, independent of $v$, such that
\begin{align*}
	[\nabla p_j^{\mathrm{nor}}]_{BMO(\mathbf{R}^n)} &\leq C_{\rho} \|v\|_{vBMO(\Omega)}, \\
	\sup_{x\in\Gamma, r<\rho} \, \frac{1}{r^n} \int_{B_r(x)} |\nabla d \cdot \nabla p_j^{\mathrm{nor}}| \, dy &\leq C_{\rho} \|v\|_{vBMO(\Omega)}.
\end{align*}
\end{proposition}

Once these two propositions are proved, we are able to prove Theorem \ref{CSV}.
\begin{proof}[Theorem \ref{CSV} admitting Proposition \ref{VPT} and \ref{2LP}]
Fix $1 \leq j \leq m$. Let us first consider the contribution from the tangential part. We take a cut-off function $\theta_j\in C^\infty_c(U_{2\rho,j})$ such that $\theta_j=1$ on $U_{\rho,j}$ and $0 \leq \theta_j \leq 1$.
We next set
\[
q^{\mathrm{tan}}_{1,j} := \theta_j p_j^{\mathrm{tan}} + E * \left(p_j^{\mathrm{tan}} \Delta \theta_j + 2\nabla\theta_j \cdot \nabla p_j^{\mathrm{tan}} \right).
\]
By definition, Proposition \ref{VPT} says that
\begin{align*}
- \Delta q^{\mathrm{tan}}_{1,j} &= - \Delta(\theta_j p_j^{\mathrm{tan}}) + p_j^{\mathrm{tan}} \Delta\theta_j + 2\nabla\theta_j \cdot \nabla p_j^{\mathrm{tan}} \\
&= \theta_j \operatorname{div} w_j^{\mathrm{tan}} = \operatorname{div} w_j^{\mathrm{tan}}
\end{align*}
in $\Omega$ as $\operatorname{supp}w_j^{\mathrm{tan}} \subset U_{\rho,j}$.
By interpolation as in the proof of Proposition \ref{VPT}, we observe that $\|p_j^{\mathrm{tan}}\|_{L^\infty(\mathbf{R}^n)}$, $\|\nabla p_j^{\mathrm{tan}}\|_{L^p(\mathbf{R}^n)}$ are controlled by $\|v\|_{BMO(\Omega)}$.
Since $\nabla E$ is in $L^{p'}(B_R)$ for $p'<n/(n-1)$ where $R = \operatorname{diam}\Omega + 4\rho$, it follows that
\[
	\sup_{\mathbf{R}^n} | \nabla E* ( p_j^{\mathrm{tan}} \Delta \theta_j + 2 \nabla \theta_j \cdot \nabla p_j^{\mathrm{tan}} ) |
	\leq C_\rho \|v\|_{vBMO(\Omega)}.
\]
Thus, by Proposition \ref{VPT}, we conclude that the restriction of $q^{\mathrm{tan}}_{1,j}$ on $\Omega$, which is still denoted by $q^{\mathrm{tan}}_{1,j}$, fulfills
\begin{equation} \label{CES}
	\| \nabla q^{\mathrm{tan}}_{1,j} \|_{vBMO(\Omega)} \leq C_\rho \|v\|_{vBMO(\Omega)}.
\end{equation}
By Proposition \ref{2LP}, a similar argument yields an estimate of type (\ref{CES}) for
\[
q^{\mathrm{nor}}_{1,j} := \theta_j p_j^{\mathrm{nor}} + E * (p_j^{\mathrm{nor}} \Delta \theta_j + 2\nabla\theta_j \cdot \nabla p_j^{\mathrm{nor}} ).
\]

Set
\[
	q^2_1 = \sum^m_{j=1} q^{\mathrm{tan}}_{1,j}, \; q_1^3 = \sum_{j=1}^m q_{1,j}^{\mathrm{nor}}, \;
	q_1 = q_1^1 + q_1^2 + q_1^3.
\]
Observe that $q_1^2$ and $q^3_1$ satisfy the desired estimates in Theorem \ref{CSV}.
Moreover, by construction we have that
\begin{align*}
	- \Delta q_1 &= - \Delta q^1_1 - \Delta q^2_1 - \Delta q^3_1 \\
	&= \operatorname{div}v_1 + \sum_{j=1}^m \operatorname{div} w_j^{\mathrm{tan}} + \sum^m_{j=1} \operatorname{div} w_j^{\mathrm{nor}} \\
	&= \operatorname{div} (v_1 + v_2) = \operatorname{div} v 
\end{align*}
in $\Omega$.
\end{proof}
%

\section{Volume potentials based on normal coordinates} 
\label{sec:2}

Our goal in this section is to prove Proposition \ref{VPT} and Proposition \ref{2LP}.
We write the Laplace operator by a normal coordinate system and construct a volume potential keeping the parity of functions with respect to the boundary.
For this purpose, we adjust a perturbation method called a freezing coefficient method which is often used to construct a fundamental solution to an operator with variable coefficients. 

\subsection{A perturbation method keeping parity} \label{3PMK} 

We consider an elliptic operator of the form
\[
	L_0 = A-B, \quad A = -\Delta_\eta, \quad B = \sum_{1\leq i,j \leq n-1} \partial_{\eta_i} b_{ij}(\eta) \partial_{\eta_j}
\]
in a cylinder $V_{4 \rho}$.
We assume that
\begin{enumerate}
\item[(B1)] (Regularity)
 $b_{ij} \in \operatorname{Lip}(V_{4 \rho})$ ($1\leq i,j \leq n-1$),
\item[(B2)] (Parity)
 $b_{ij}$ is even in $\eta_n$, i.e., $b_{ij}(\eta', \eta_n)=b_{ij}(\eta', - \eta_n)$ for $\eta \in V_{4 \rho}$,
\item[(B3)] (Smallness)
 $b_{ij}(0)=0$ ($1\leq i,j\leq n-1$).
\end{enumerate}
For $\rho > 0$, let $Y_\rho$ denotes the space
\[
	\left\{ g \in BMO(\mathbf{R}^n) \cap L^2(\mathbf{R}^n) \bigm|
	\operatorname{supp} g \subset V_\rho,\ 
	g(\eta', \eta_n) = g(\eta', -\eta_n)\ \text{for}\ \eta \in V_\rho \right\},
\] 
whereas $Z_\rho$ denotes the space
\[
	\left\{ f \in BMO(\mathbf{R}^n) \bigm|
	\operatorname{supp} f \subset V_\rho,\ 
	f(\eta', \eta_n) = -f(\eta', -\eta_n)\ \text{for}\ \eta \in V_\rho \right\}.
\]
The oddness condition in $Z_\rho$ guarantees that
\[
	\frac{1}{r^n} \int_{B_r(\eta',0)} f \, d\eta = 0
\]
for any $r > 0$ and $\eta' \in \mathbf{R}^{n-1}$, which implies that
\[
	\frac{1}{r^n} \int_{B_r(\eta',0)} |f| \, d\eta \leq [f]_{BMO(\mathbf{R}^n)}
\]
for any $r > 0$ and $\eta' \in \mathbf{R}^{n-1}$. Hence $f$ is $L^1$ in $\mathbf{R}^n$.
\begin{lemma} \label{3P}
Assume that (B1) -- (B3).
Then, there exists $\rho_*>0$ depending only on $n$ and $b$ such that the following property holds provided that $\rho\in(0,\rho_*)$.
There exists a bounded linear operator $f \mapsto q_o$ from $Z_\rho$ to $L^\infty(\mathbf{R}^n)$ such that
\begin{enumerate}
\item[(i)]
\begin{align*}
&[ \nabla_\eta q_o ]_{BMO(\mathbf{R}^n)} \leq C[f]_{BMO(\mathbf{R}^n)} \quad\text{for all}\quad f\in Z_\rho
\end{align*}
with some $C$ independent of $f$;
\item[(ii)]
\begin{align*}
&L_0 q_o = \partial_{\eta_n} f \quad\text{in}\quad V_{2 \rho};
\end{align*}
\item[(iii)] $q_o$ is even in $\mathbf{R}^n$ with respect to $\eta_n$, i.e. $q_o(\eta',\eta_n)=q_o(\eta',-\eta_n) \; \forall \, \eta \in \mathbf{R}^n$;
\item[(iv)]\[
	\sup \left\{ \frac{1}{r^n} \int_{B_r(\eta',0)} \left|\partial_{\eta_n} q_o \right| \, d\eta \biggm|
	0< r < \infty,\, \eta' \in \mathbf{R}^{n-1} \right\}
	\leq C[f]_{BMO(\mathbf{R}^n)}.
\]
\end{enumerate}
\end{lemma}
\begin{proof}
By (B3) and (B1), we observe that
\[
	\varlimsup_{\rho\downarrow 0} \, \|b_{ij}\|_{C^\gamma(V_{4 \rho})} \bigm/ \rho^{1-\gamma} < \infty
\]
for any $\gamma\in(0,1)$ and $1 \leq i,j \leq n-1$.
Indeed, for $1 \leq i,j \leq n-1$, (B1) and (B3) imply that
\begin{align*}
	\|b_{ij}\|_{L^\infty(V_{4 \rho})} &\leq 8 L\rho, \\
	[b_{ij}]_{C^\gamma(V_{4 \rho})} &:= \sup \left\{ | b_{ij}(\eta) - b_{ij}(\zeta) | \bigm/ | \eta - \zeta |^\gamma \Bigm| \eta,\zeta \in V_{4 \rho} \right\} \\
	&\leq L(16 \rho)^{1-\gamma},
\end{align*}
where $L$ is the maximum of Lipschitz bound for $b_{ij}$ for all $1\leq i,j\leq n-1$.
We next take a cut-off function.
 We take $\theta\in C^\infty_c(V_4)$ such that $\theta=1$ on $V_2$ and $0\leq\theta\leq 1$ in $V_4$,
 we may assume $\theta$ is radial so that $\theta$ is even in $\eta_n$.
 We rescale $\theta$ by setting
\[
	\theta_\rho(\eta) = \theta(\eta/\rho)
\]
so that $\theta_\rho=1$ on $V_{2 \rho}$.
Since $\| \nabla \theta_\rho \|_\infty \rho$ is bounded as $\rho\to 0$, we see that
\[
	\varlimsup_{\rho\downarrow 0} \, [\theta_\rho]_{C^\gamma(V_{4 \rho})} \rho^\gamma < \infty.
\] 
Hence, the estimate
\[
	[\theta_\rho b_{ij}]_{C^\gamma(V_{4 \rho})}
	\leq [\theta_\rho]_{C^\gamma(V_{4 \rho})} \|b_{ij}\|_{L^\infty(V_{4 \rho})}
	+ [b_{ij}]_{C^\gamma(V_{4 \rho})} \|\theta_\rho\|_{L^\infty(V_{4 \rho})}
\]
implies that
\[
	\varlimsup_{\rho\downarrow 0} \, \|\theta_\rho b_{ij}\|_{C^\gamma(V_{4 \rho})} \bigm/ \rho^{1-\gamma} < \infty.
\]

We then set
\[
	L_1 = A - B_1, \quad
	B_1 = \sum_{1\leq i,j \leq n-1} \partial_{\eta_i} b^\rho_{ij} \partial_{\eta_j}, \quad
	b^\rho_{ij} = b_{ij} \theta_\rho.
\]
For $1 \leq i,j \leq n-1$, notice that $b^\rho_{ij}$ satisfies the same property of $b_{ij}$ in (B1) -- (B3).
Moreover,
\[
	\operatorname{supp} \, b^\rho_{ij} \subset V_{4 \rho} \quad\text{and}\quad
	\left\| b^\rho_{ij} \right\|_{C^\gamma(V_{4 \rho})} \leq c_b \rho^{1-\gamma},\ \rho > 0
\]
with some $c_b$ independent of $\rho$. Since $\operatorname{supp} \, b^\rho_{ij} \subset V_{4 \rho}$, we actually have that $b_{ij}^\rho \in C^\gamma(\mathbf{R}^n)$ together with the estimate
\[
\| b_{ij}^\rho \|_{C^\gamma(\mathbf{R}^n)} \leq \| b_{ij}^\rho \|_{C^\gamma(V_{4 \rho})}.
\]
For a given $f \in Z_\rho$, we define $q_o$ by
\[
	q_o := \sum^\infty_{k=0} A^{-1} \left( B_1 A^{-1} \right)^k \partial_{\eta_n} f,
\]
where formally for a function $h$ we mean $A^{-1} h$ by $E*h$.
The parity condition (iii) is clear once $q_o$ is well defined as a function.
Since
\[
	L_1 q_o = \sum^\infty_{k=0} \left( B_1 A^{-1} \right)^k \partial_{\eta_n} f
	- \sum^\infty_{k=1} \left( B_1 A^{-1} \right)^k \partial_{\eta_n} f = \partial_{\eta_n}f 
\]
in $\mathbf{R}^n$, the property (ii) then follows since $L_1=L_0$ in $V_{2 \rho}$.

It remains to prove the convergence of $q_o$ as well as (i).
 For this purpose, we reinterpret $q_o$ in a different way.
 We rewrite
\[
	B _1 = \operatorname{div}' \cdot \nabla'_B \quad\text{with}\quad
	\nabla'_B = \left( \sum^{n-1}_{j=1} b^\rho_{ij} \partial_{\eta_j} \right)_{1\leq i \leq n-1}
\]
and observe that
\begin{align*}
	q_o &= \sum^\infty_{k=0} A^{-1} \operatorname{div}' \cdot G^k \cdot \nabla'_B A^{-1} \partial_{\eta_n} f
	+ A^{-1} \partial_{\eta_n} f, \\
	G &:= \nabla'_B A^{-1} \operatorname{div}'.
\end{align*}
Denote
\[
b^\rho := \left( b_{ij}^\rho \right)_{1 \leq i,j \leq n-1}.
\]
Since $\partial_{\eta_\alpha} A^{-1}\partial_{\eta_\beta}$ is bounded in $BMO$ \cite{FS} and also in $L^p$ ($1<p<\infty$) for all $\alpha, \beta=1,\ldots,n$, see e.g. \cite[Theorem 5.2.7 and Theorem 5.2.10]{Gra}, by a multiplication theorem we can deduce the estimates
\begin{align}
	\| Gh \|_{L^p(\mathbf{R}^n)} &\leq C_p \| b^\rho \|_{L^\infty(\mathbf{R}^n)} \| h \|_{L^p(\mathbf{R}^n)}, \label{PEG} \\
	[Gh]_{BMO(\mathbf{R}^n)} &\leq C'_\infty \| b^\rho \|_{C^\gamma(\mathbf{R}^n)}
	\left( [h]_{BMO(\mathbf{R}^n)} + \| h \|_{L^1(\mathbf{R}^n)} \right) \label{BEG}
\end{align}
provided that $\operatorname{supp} \, h \subset V_{4 \rho}$ and $\rho<1$.
 Here $C_p$ and $C'_\infty$ are independent of $\rho$ and $h$.
 Similar estimate holds for $\nabla'_B A^{-1}\partial_{\eta_n}$.
 Since $\|f\|_{L^1(\mathbf{R}^n)} \leq C_\rho [f]_{BMO(\mathbf{R}^n)}$ for $f\in Z_\rho$, by an interpolation (cf. \cite[Lemma 5]{BGST}) we see that the $L^p$ norm of $f$ is also controlled, i.e., $\|f\|_{L^p(\mathbf{R}^n)} \leq C_\rho [f]_{BMO(\mathbf{R}^n)}$ for any $1\leq p<\infty$.
 By the support condition, $A^{-1}\operatorname{div}'$ and $A^{-1} \partial_{\eta_n}$ is bounded from $L^p\to L^\infty$ for $p>n$ with bound $K$, we see that
\begin{align*}
	\| q_o \|_{L^\infty(\mathbf{R}^n)} &\leq K \left( \left\| \sum^\infty_{k=0} G^k \nabla'_B A^{-1}\partial_{\eta_n} f \right\|_{L^p(\mathbf{R}^n)} + \|f\|_{L^p(\mathbf{R}^n)} \right) \\
	&\leq K \left( \sum^\infty_{k=0} C^{k+1}_p \| b^\rho \|^{k+1}_{L^\infty(\mathbf{R}^n)} \|f\|_{L^p(\mathbf{R}^n)} + \|f\|_{L^p(\mathbf{R}^n)} \right), \quad p>n.
\end{align*}
If $\rho$ is taken small so that
\[
	\sum^\infty_{k=0} (C_p \cdot 8 L \rho)^{k+1} < \infty,
\]
then $q_o$ converges uniformly in $\mathbf{R}^n$ and $\| q_o \|_{L^\infty(\mathbf{R}^n)} \leq C_\rho [f]_{BMO(\mathbf{R}^n)}$ with some $C_\rho$ independent of $f$.

Set
\[
	\| h \|_{BMOL^p(\mathbf{R}^n)} := [h]_{BMO(\mathbf{R}^n)} + \| h \|_{L^p(\mathbf{R}^n)}.
\]
By estimates (\ref{PEG}) and (\ref{BEG}), we observe that
\[
	\| G h \|_{BMOL^p(\mathbf{R}^n)} \leq C_\ast \| b^\rho \|_{C^\gamma(\mathbf{R}^n)} \| h \|_{BMOL^p(\mathbf{R}^n)},
	\quad 1<p<\infty,
\]
where $C_\ast = C_p + C'_\infty \cdot C_n$ with $C_n$ independent of $\rho$ and $h$.
We next estimate $\nabla q_o$.
By the similar estimate for $\nabla'_B A^{-1}\operatorname{div}'$ and $\nabla'_B A^{-1} \partial_{\eta_n}$, we have that
\[
	\| \nabla q_o \|_{BMOL^p(\mathbf{R}^n)} \leq \left( \sum^\infty_{k=0} C^{k+1}_\ast \| b^\rho \|^{k+1}_{C^\gamma(\mathbf{R}^n)} + C_* \| b^\rho \|_{C^\gamma(\mathbf{R}^n)} \right) \|f\|_{BMOL^p(\mathbf{R}^n)}.
\]
We fix $p>n$ and take $\rho < \frac{1}{8 L C_p}$ sufficiently small so that
\[
	\sum^\infty_{k=0} \left( C_*  \cdot c_b \rho^{1-\gamma} \right)^{k+1} < \infty.
\]
Then we get our desired estimate
\[
	\| \nabla q_o \|_{BMOL^p (\mathbf{R}^n)} \leq C_\rho \|f\|_{BMOL^p(\mathbf{R}^n)} \leq C_\rho [f]_{BMO(\mathbf{R}^n)}
\]
for $f \in Z_\rho$.
This completes the proof of (i).

Since $\partial_{\eta_n} q_o$ is odd in $\eta_n$ so that
\[
	\frac{1}{r^n} \int_{B_r(\eta',0)} \partial_{\eta_n} q_o \, d\eta = 0
\]
for any $\eta' \in \mathbf{R}^{n-1}$,
the left-hand side of (iv) is estimated by a constant multiple of $[\partial_{\eta_n} q_o]_{BMO(\mathbf{R}^n)}$, which is estimated by a constant multiple of $[f]_{BMO(\mathbf{R}^n)}$.
The proof of (iv) is now complete.
\end{proof}

Similarly, we are able to establish the following which corresponds to a version of Lemma \ref{3P} for the space $Y_\rho$.

\begin{lemma} \label{4P}
Assume that (B1) -- (B3). Then, there exists $\rho_*>0$ depending only on $n$ and $b$ such that the following property holds provided that $\rho\in(0,\rho_*)$. For each $1 \leq i \leq n-1$, there exists a bounded linear operator $g \mapsto q_{e,i}$ from $Y_\rho$ to $L^\infty(\mathbf{R}^n)$ such that
\begin{enumerate}
\item[(i)]
\begin{align*}
&[ \nabla q_{e,i} ]_{BMO(\mathbf{R}^n)} \leq C \|g\|_{BMOL^2(\mathbf{R}^n)} \quad \text{for all} \quad g \in Y_\rho
\end{align*}
with some $C$ independent of $f$;
\item[(ii)]
\begin{align*}
&L_0 q_{e,i} = \partial_{\eta_i} g \quad \text{in} \quad V_{2 \rho};
\end{align*}
\item[(iii)] $q_{e,i}$ is even in $\mathbf{R}^n$ with respect to $\eta_n$, i.e. $q_{e,i}(\eta',\eta_n)=q_{e,i}(\eta',-\eta_n) \; \forall \, \eta \in \mathbf{R}^n$;
\item[(iv)]
\[
	\sup \left\{ \frac{1}{r^n} \int_{B_r(\eta',0)} \left|\partial_{\eta_n} q_{e,i} \right| \, d\eta \biggm|
	0 < r < \infty,\, \eta' \in \mathbf{R}^{n-1}  \right\}
	\leq C \|g\|_{BMOL^2(\mathbf{R}^n)}.
\]
\end{enumerate}
\end{lemma}

\begin{proof}
Fix $1 \leq i \leq n-1$. Since $g$ is even in $\mathbf{R}^n$ with respect to $\eta_n$, $\partial_{\eta_i} g$ is also even in $\mathbf{R}^n$ with respect to $\eta_n$. This means that $\partial_{\eta_i} g$ has the same parity with $\partial_{\eta_n} f$ in Lemma \ref{3P}. By considering
\[
q_{e,i} := \sum_{k=0}^\infty A^{-1} (B_1 A^{-1})^k \partial_{\eta_i} g,
\]
exactly the same arguments of the proof of Lemma \ref{3P} finish the rest of the work.
\end{proof}

We take $\rho_\ast$ in Lemma \ref{3P} and Lemma \ref{4P} to be 
\[
\rho_\ast := \mathrm{min} \, \left\{ \rho_{0,\ast}, \; \frac{1}{8 L C_p}, \; \bigg( \frac{1}{C_\ast \cdot c_b} \bigg)^{\frac{1}{1-\gamma}} \right\}.
\]

\subsection{Laplacian in a normal coordinate system} \label{3LN} 

Take $z_0 \in \Gamma$. Let us recall the normal coordinate system (\ref{NCC}) introduced in Section \ref{LOC}, i.e.,
\begin{eqnarray*}
\left\{
\begin{array}{lcl}
	x' &=& \eta' + \eta_n \nabla' d ( \eta', h_{z_0}(\eta') ); \\
	x_n &=& h_{z_0}(\eta') + \eta_n \partial_{\eta_n} d ( \eta', h_{z_0}(\eta') )
\end{array}
\right.
\end{eqnarray*}
in $U_{\rho_0}(z_0)$ with $\nabla' h_{z_0}(0')=0$, $h_{z_0}(0')=0$ up to translation and rotation such that $z_0=0$ and 
\[
	-\mathbf{n} \left( \eta', h_{z_0}(\eta') \right)
	= \left( -\nabla' h_{z_0}(\eta'), 1 \right) \Bigm/
	\left( 1 + \left| \nabla'_{z_0} h( \eta' )\right|^2 \right)^{1/2},
	\quad \eta' \in B_{\rho_0}.
\]
Since $\Gamma$ is $C^3$, the mapping $x=\psi(\eta) \in C^2(V_{\rho_0})$ in $U_{\rho_0}(z_0)$, it is a local $C^2$-diffeomorphism.
Moreover, its Jacobi matrix $D\psi$ is the identity at $0$, i.e.,
\[
	\nabla \psi (0) = I = \nabla \psi^{-1} (0).
\]
A direct calculation shows that in $U_{\rho_0}(z_0) \cap \Omega$,
\begin{align*}
	-\Delta_x = -\Delta_{\eta} &- \left\{ \sum_{\substack{1\leq i,j \leq n-1 \\ i\neq j}} \gamma_{ij} \partial_{\eta_i} \partial_{\eta_j}
	+ \sum^{n-1}_{j=1} (\gamma_{jj}-1) \partial^2_{\eta_j} \right\} \\
	&- \sum_{1\leq i,j \leq n} \frac{\partial^2 \eta_j}{\partial x^2_i} \partial_{\eta_j}, \ 
	 \gamma_{ij} = \sum^n_{k=1} \frac{\partial \eta_j}{\partial x_k} \frac{\partial \eta_i}{\partial x_k}.
\end{align*}
Note that $\gamma_{jj}(0)=1$ while $\gamma_{ij}(0)=0$ if $i \neq j$.
Changing order of multiplication and differentiation, we conclude that
\begin{align*}
	-\Delta_x &= \tilde{L}_0 + \tilde{M}, \\
	\tilde{L}_0 &:= A - \tilde{B},\quad A := -\Delta_\eta,\quad \tilde{B} := \sum_{1\leq i,j \leq n-1}
	\partial_{\eta_i} \tilde{b}_{ij} (\eta) \partial_{\eta_j}, \\
	\tilde{M} &:= \sum^n_{j=1} \tilde{c}_j (\eta) \partial_{\eta_j}
\end{align*}
with $\tilde{b}_{ij}=\gamma_{ij}-\delta_{ij}$, $\tilde{c}_j=-\sum^n_{i=1} \frac{\partial^2 \eta_j}{\partial x^2_i} + \sum^n_{i=1} \partial_{\eta_i} \gamma_{ij}$.
Note that if $\Gamma=\partial\Omega$ is $C^3$, $\tilde{b}_{ij} \in C^1(V_{\rho_0})$ and $\tilde{c_j} \in C(V_{\rho_0})$.
We restrict $\tilde{b}_{ij}$, $\tilde{c}_j$ in $V_{\rho_0} \cap \mathbf{R}_+^n$ and extend  to $V_{\rho_0}$ so that the extended function $b_{ij}$, $c_j$'s are even in $V_{\rho_0}$ with respect to $\eta_n$, i.e., we set $b_{ij} = E_{\mathrm{even}} \; r_{V_{\rho_0} \cap \mathbf{R}_+^n} \, \tilde{b_{ij}}$ and $c_j = E_{\mathrm{even}} \; r_{V_{\rho_0} \cap \mathbf{R}_+^n} \, \tilde{c_j}$. 
By this extension, $b_{ij}$ may not be in $C^1$ but still Lipschitz and $c_j\in C (V_{\rho_0})$.
We set
\begin{align*}
	B &:= \sum_{1\leq i,j \leq n-1} \partial_{\eta_i} b_{ij} (\eta) \partial_{\eta_j}, \\
	M &:= \sum^n_{j=1} c_j (\eta) \partial_{\eta_j}
\end{align*}
and
\[
	L := L_0 + M,\quad L_0 = A - B.
\]
The operator $L$ may not agree with $-\Delta_x$ outside $U_{\rho_0}(z_0) \cap \Omega$.
We summarize what we observe so far.
\begin{proposition} \label{3R}
Let $\Gamma=\partial\Omega$ be $C^3$ and $\rho_0$ be chosen as in Section \ref{LOC}.
For $z_0 \in \Gamma$, $L_0$ satisfies (B1) -- (B3).
 Moreover, $-\Delta_x=L$ in $U_{\rho_0}(z_0) \cap \Omega$ and the coefficient of $M$ is in $C(V_{\rho_0})$.
\end{proposition}

Although we do not use the explicit formula of $\Delta$ in normal coordinates, we give it for $n=2$ when we take the arc length parameter $s$ to represent $\Gamma$.
 The coordinate transform is of the form
\begin{align*}
	x_1 &= \phi_1(x) + r\phi'_2 (s) \\
	x_2 &= \phi_2(x) - r\phi'_1 (s) 
\end{align*}
with $\phi'^2_1 + \phi'^2_2=1$ and $r=d(x)$.
 A direct calculation yields
\[
	-\Delta_x = -\Delta_{s,r} -\partial_s \left( \frac{1}{J^2}-1 \right) \partial_s
	- \frac{\partial_s J}{J^3} \partial_s
	- \frac{1}{r} \left( 1-\frac{1}{J} \right) \partial_r,
\]
where $J=1+r\kappa$ and $\kappa$ is the curvature.
We see that that the even extension of coefficient does not agree with $-\Delta_x$ outside $\Omega$.
\subsection{$bmo$ invariant under local $C^1$-diffeomorphism} \label{BMOI}
Before we give the proofs to Proposition \ref{VPT} and \ref{2LP}, we shall first establish the fact that the $bmo$ estimate of a compactly supported function is preserved under a local $C^1$-diffeomorphism. Let $V,U \subset \mathbf{R}^n$ be two domains, we consider a local $C^1$-diffeomorphism $\psi: V \mapsto U$. Suppose that
\[
\| \nabla_\eta \psi \|_{L^\infty(V)} + \| \nabla_x \psi^{-1} \|_{L^\infty(U)} < \infty.
\]
 Let $\rho>0$. Assume that there exist two bounded subdomains $V_\rho \subset V, U_\rho \subset U$ such that $\psi:V_\rho \mapsto U_\rho$ is also a local $C^1$-diffeomorphism. Set
\[
K_\ast := \mathrm{max} \left\{1, \| \nabla_\eta \psi \|_{L^\infty(V)} + \| \nabla_x \psi^{-1} \|_{L^\infty(U)} \right\}.
\]
We assume further that there exists a constant $c_0$ such that for some $\eta_0 \in V_\rho$,
\[
V_\rho \subset B_{c_0 \rho}(\eta_0) \subset B_{K_\ast (c_0+3) \rho}(\eta_0) \subset V, \; \; U_\rho \subset B_{c_0 \rho}(x_0) \subset B_{K_\ast (c_0+3) \rho}(x_0) \subset U
\]
where $x_0 = \psi(\eta_0)$.

\begin{proposition} \label{BNCS}
Let $f \in bmo(\mathbf{R}^n)$ with $\operatorname{supp} f \subset V_\rho$, then $f \circ \psi^{-1} \in bmo(\mathbf{R}^n)$ satisfies
\[
\| f \circ \psi^{-1} \|_{bmo(\mathbf{R}^n)} \leq C_\rho \| f \|_{bmo(\mathbf{R}^n)}.
\]
\end{proposition}

\begin{proof}
Since $\operatorname{supp} f \circ \psi^{-1} \subset U_\rho$, we can treat $f \circ \psi^{-1}$ as a function in $\mathbf{R}^n$ with value zero outside $U_\rho$. The compactness of $V_\rho$ in $\mathbf{R}^n$ implies that $\| f \|_{bmo(\mathbf{R}^n)} = \| f \|_{BMOL^1(\mathbf{R}^n)}$. Thus, the $L^1$ estimate 
\[
\| f \circ \psi^{-1} \|_{L^1(\mathbf{R}^n)} \leq C \| f \|_{L^1(\mathbf{R}^n)}
\]
is obvious. Since $\psi \in C^1(V_\rho)$, an equivalent definition of the $BMO$-seminorm (cf. \cite[Proposition 3.1.2]{Gra}) implies that
\[
[f \circ \psi^{-1}]_{BMO^\infty(B_{( c_0+1) \rho}(x_0))} \leq \| \nabla_x \psi^{-1} \|_{L^\infty(U)}^n \cdot \| \nabla_\eta \psi \|_{L^\infty(V)} \cdot [f]_{BMO(\mathbf{R}^n)}.
\]
As $U_\rho \subset B_{c_0 \rho}(x_0)$, by the extension theorem of $bmo$ functions \cite[Theorem 12]{GigaGu2}, we obtain that
\[
\| f \circ \psi^{-1} \|_{bmo(\mathbf{R}^n)} \leq C_\rho \|f \circ \psi^{-1}\|_{bmo_\infty^\infty(B_{(c_0+1) \rho}(x_0))} \leq C_\rho \| f \|_{bmo(\mathbf{R}^n)}.
\]
\end{proof}

Similarly, if $g \in bmo(\mathbf{R}^n)$ with $\operatorname{supp} g \subset U_\rho$, then we have that $g \circ \psi \in bmo(\mathbf{R}^n)$ satisfying
\[
\| g \circ \psi \|_{bmo(\mathbf{R}^n)} \leq C_\rho \| g \|_{bmo(\mathbf{R}^n)}.
\]
Even if we are considering vector fields instead of scalar functions, similar results hold.

\begin{proposition} \label{BNCV}
Let $\nabla_\eta f \in bmo(\mathbf{R}^n)$ with $\operatorname{supp} \nabla_\eta f \subset V_\rho$, then $\nabla_x (f \circ \psi^{-1}) \in bmo(\mathbf{R}^n)$ satisfying
\[
\| \nabla_x (f \circ \psi^{-1}) \|_{bmo(\mathbf{R}^n)} \leq C_\rho \| \nabla_\eta f \|_{bmo(\mathbf{R}^n)}.
\]
\end{proposition}

\begin{proof}
Since $\nabla_\eta f$ is compactly supported, the $L^1$ estimate
\[
\| \nabla_x (f \circ \psi^{-1}) \|_{L^1(\mathbf{R}^n)} \leq C \| \nabla_\eta f \|_{L^1(\mathbf{R}^n)}
\]
is obvious. Pick a cut-off function $\theta_{\ast,\rho} \in C_c^\infty(B_{K_\ast (c_0+3) \rho}(\eta_0))$ such that $\theta_{\ast,\rho} = 1$ in $B_{K_\ast (c_0+2) \rho}(\eta_0)$. Consider $B_r(x) \subset B_{(c_0+1) \rho}(x_0)$ with $r < \rho$. Let $\eta = \psi^{-1}(x)$. Since $\psi^{-1}(B_r(x)) \subset B_{K_\ast (c_0+2) \rho}(\eta_0)$, we have that
\begin{align*}
\frac{1}{r^n} \int_{B_r(x)} | \partial_{x_i} (f \circ \psi^{-1}) - c | \, dy &\leq \frac{K_\ast}{r^n} \int_{\psi^{-1}(B_r(x))} \bigg| \sum_{1 \leq l \leq n} \theta_{\ast,\rho} \bigg( \frac{\partial \eta_l}{\partial x_i} \bigg)_\psi \partial_{\eta_l} f - c \bigg| \, d\eta 
\end{align*}
for any $c \in \mathbf{R}^n$, $1 \leq i \leq n$. By considering an equivalent definition of the $BMO$-seminorm, see e.g. \cite[Proposition 3.1.2]{Gra}, we deduce that
\begin{align*}
[\nabla_x (f \circ \psi^{-1})]_{BMO^\infty(B_{(c_0+1) \rho}(x_0))} &\leq K_\ast^{n+1} \bigg[ \sum_{1 \leq i,l \leq n} \theta_{\ast,\rho} \bigg( \frac{\partial \eta_l}{\partial x_i} \bigg)_\psi \partial_{\eta_l} f \bigg]_{BMO(\mathbf{R}^n)} \\
&\leq C_\rho \| \nabla_\eta f \|_{bmo(\mathbf{R}^n)}.
\end{align*}
As $U_\rho \subset B_{c_0 \rho}(x_0)$, by the extension theorem of $bmo$ functions \cite[Theorem 12]{GigaGu2}, we obtain that
\[
\| \nabla_x (f \circ \psi^{-1}) \|_{bmo(\mathbf{R}^n)} \leq C_\rho \| \nabla_x (f \circ \psi^{-1}) \|_{bmo_\infty^\infty(B_{(c_0+1) \rho}(x_0))} \leq C_\rho \| \nabla_\eta f \|_{bmo(\mathbf{R}^n)}.
\]
\end{proof}

If $\nabla_x g \in bmo(\mathbf{R}^n)$ with $\operatorname{supp} \nabla_x g \subset U_\rho$, same proof of Proposition \ref{BNCV} shows that $\nabla_\eta (g \circ \psi) \in bmo(\mathbf{R}^n)$ satisfying
\[
\| \nabla_\eta (g \circ \psi) \|_{bmo(\mathbf{R}^n)} \leq C_\rho \| \nabla_x g \|_{bmo(\mathbf{R}^n)}.
\]

Let $h$ be either a scalar function or a vector field which is compactly supported in $U_\rho$, for simplicity of notations we denote $h_\psi := h \circ \psi$. If $h$ is a vector field, we denote $h_{\psi,i} := h_i \circ \psi$ for $1 \leq i \leq n$.
\subsection{Volume potential for tangential component} \label{3VPT}
Let $\rho \in (0,\rho_\ast/2)$ and fix $1 \leq j \leq m$. 
Since $\varphi_j v_2 \in vBMO(\Omega)$ with $\operatorname{supp} \,\varphi_j v_2 \subset U_{\rho,j} \cap \overline{\Omega}$, Proposition \ref{2E} implies that $(\varphi_j v_2)_{\mathrm{even}} \in BMOL^1(\mathbf{R}^n)$. By the product estimate for $bmo$ functions \cite[Theorem 13]{GigaGu2}, we see that $w_j^{\mathrm{tan}} = Q (\varphi_j v_2) \in BMOL^1(\mathbf{R}^n)$ with $\operatorname{supp} w_j^{\mathrm{tan}} \subset U_{\rho,j}$. For simplicity of notations, we set $v_{2,j} := (\varphi_j v_2)_{\mathrm{even}}$.

Let $\psi: V_{4 \rho} \mapsto U_{4 \rho,j}$ be the normal coordinate change defined by (\ref{NCC}) in Section \ref{LOC}.
Since $\rho < \rho_\ast/2$, we have that
\[
V_{4 \rho} \subset B_{12 \rho}(0) \subset B_{24 L_\ast \rho}(0) \subset V_{\rho_0}, \; \; U_{4 \rho,j} \subset B_{12 \rho}(z_j) \subset B_{24 L_\ast \rho}(z_j) \subset U_{\rho_0,j}.
\]
By Proposition \ref{BNCS} and \ref{BNCV}, we see that $\psi$, in this case, is a local $C^2$-diffeomorphism that preserves $bmo$ estimates for functions or vector fields compactly supported in $V_{4 \rho}$. As a result, $(v_{2,j})_\psi \in BMOL^1(\mathbf{R}^n)$ satisfies the estimate
\[
\| (v_{2,j})_\psi \|_{BMOL^1(\mathbf{R}^n)} \leq C_\rho \| v_{2,j} \|_{BMOL^1(\mathbf{R}^n)}.
\]
Note that similar conclusions hold if we consider $\psi^{-1}: U_{4 \rho,j} \mapsto V_{4 \rho}$ instead.

\begin{proof}[Proposition \ref{VPT}]
For $1 \leq i \leq n$ and $1 \leq k \leq n-1$, we define
\[
\bigg( \frac{\partial \eta_k}{\partial x_i} \bigg)_\ast := E_{\mathrm{even}} \; r_{V_{4 \rho} \cap \mathbf{R}_{+}^n} \; \bigg( \frac{\partial \eta_k}{\partial x_i} \bigg)_\psi \, \; \text{ and } \, \; g_{i,k} := \bigg( \frac{\partial \eta_k}{\partial x_i} \bigg)_\ast \cdot (v_{2,j})_{\psi,i}.
\]
We consider 
\begin{align*}
( \operatorname{div}_x w_j^{\mathrm{tan}} )_{\psi,\ast} &:= \underset{1 \leq k \leq n-1}{\sum_{1 \leq i \leq n,}} \left\{ \partial_{\eta_k} g_{i,k} -  \partial_{\eta_k} \bigg( \frac{\partial \eta_k}{\partial x_i} \bigg)_\psi \cdot (v_{2,j})_{\psi,i} \right\}\\
&- \underset{1 \leq k \leq n-1}{\sum_{1 \leq i \leq n,}} \bigg( \frac{\partial \eta_k}{\partial x_i} \bigg)_\psi \cdot \left( \sum_{1 \leq l \leq n} (v_{2,j})_{\psi,l} \cdot \bigg( \frac{\partial \eta_n}{\partial x_l} \bigg)_\psi \right) \cdot \frac{\partial^2 x_i}{\partial \eta_k \partial \eta_n}
\end{align*}
in $V_{4 \rho} = \psi^{-1}(U_{4 \rho,j})$.
Let $L = L_0 + M$ be the operator in Proposition \ref{3R} and $L_0^{-1}$ be the operator in Lemma \ref{4P}. Let $1 \leq i \leq n$ and $1 \leq k \leq n-1$. We set
\[
q_{j,1,\psi}^{i,k} := - \theta_\rho L_0^{-1} \partial_{\eta_k} g_{i,k}
\]
where $\theta_\rho$ is the cut-off function defined in the proof of Lemma \ref{3P}. 
There exists $\overline{( \frac{\partial \eta_k}{\partial x_i} )_\ast} \in C^{0,1}(\mathbf{R}^n)$, see e.g. \cite[Theorem 13]{GigaGu2}, such that the restriction of $\overline{( \frac{\partial \eta_k}{\partial x_i} )_\ast}$ in $V_{4 \rho}$ equals $( \frac{\partial \eta_k}{\partial x_i} )_\ast$ and $\| \overline{( \frac{\partial \eta_k}{\partial x_i} )_\ast} \|_{C^{0,1}(\mathbf{R}^n)} \leq \| ( \frac{\partial \eta_k}{\partial x_i} )_\ast \|_{C^{0,1}(V_{4 \rho})}$. 
By viewing $g_{i,k}$ as $\overline{( \frac{\partial \eta_k}{\partial x_i} )_\ast} \cdot (v_{2,j})_{\psi,i}$, we see that $g_{i,k} \in BMOL^1(\mathbf{R}^n)$. Hence, $q_{j,1,\psi}^{i,k} \in L^\infty(\mathbf{R}^n)$ is well-defined which satisfies all conditions in Lemma \ref{4P}. Let $f_{j,1,\psi}^{i,k} := M \theta_\rho L_0^{-1} \partial_{\eta_k} g_{i,k}$. We can define
\[
q_{j,1}^{i,k} := q_{j,1,\psi}^{i,k} \circ \psi^{-1}, f_{j,1}^{i,k} := f_{j,1,\psi}^{i,k} \circ \psi^{-1}
\]
in $U_{\rho_0,j}$. Notice that $\operatorname{supp} q_{j,1}^{i,k}, \operatorname{supp} f_{j,1}^{i,k} \subset U_{4 \rho,j}$, we can indeed treat $q_{j,1}^{i,k}, f_{j,1}^{i,k}$ as functions defined in $\mathbf{R}^n$ where their values outside $U_{4 \rho,j}$ equal zero. Proposition \ref{BNCV} shows that $\nabla_x q_{j,1}^{i,k} \in BMO(\mathbf{R}^n)$ satisfies the estimate
\[
[\nabla_x q_{j,1}^{i,k}]_{BMO(\mathbf{R}^n)} \leq C_\rho \| \nabla_\eta q_{j,1,\psi}^{i,k} \|_{BMOL^2(\mathbf{R}^n)} \leq C_\rho \|g_{i,k}\|_{BMOL^2(\mathbf{R}^n)}.
\]
Let $p_{j,1}^{i,k} := E \ast f_{j,1}^{i,k}$. By Lemma \ref{4P} again, we can prove that
\[
\| p_{j,1}^{i,k} \|_{L^\infty(\mathbf{R}^n)} + \| \nabla_x p_{j,1}^{i,k} \|_{L^\infty(\mathbf{R}^n)} \leq C_\rho \| f_{j,1,\psi}^{i,k} \|_{L^p(V_{2 \rho})} \leq C_\rho \| g_{i,k} \|_{L^p(\mathbf{R}^n)}
\]
with some $p > n$. Thus, $p_{j,1}^{i,k}$ is well-defined. By Proposition \ref{2E}, we have that
\begin{align*}
\| g_{i,k} \|_{BMOL^1(\mathbf{R}^n)} &\leq C_\rho \| v_{2,j} \|_{BMOL^1(\mathbf{R}^n)} \leq C_\rho \| v \|_{vBMO(\Omega)}.
\end{align*}
Hence, by an interpolation (cf. \cite[Lemma 5]{BGST}),
\[
\| g_{i,k} \|_{L^p(\mathbf{R}^n)} \leq C_\rho \| v \|_{vBMO(\Omega)}
\]
for any $1<p<\infty$.

For lower order term $q_{j,2,\psi}^{i,k} := \partial_{\eta_k} ( \frac{\partial \eta_k}{\partial x_i} )_\psi \cdot (v_{2,j})_{\psi,i}$, we set $q_{j,2}^{i,k} := q_{j,2,\psi}^{i,k} \circ \psi^{-1}$ in $U_{\rho_0,j}$. Similar as $q_{j,1}^{i,k}$, we can treat $q_{j,2}^{i,k}$ as a function in $\mathbf{R}^n$ with value zero outside $U_{\rho,j}$ since $\operatorname{supp} q_{j,2}^{i,k} \subset U_{\rho,j}$. Define $p_{j,2}^{i,k} := E \ast q_{j,2}^{i,k}$. Since $E$ and $\nabla_x E$ are locally integrable, we have that
\[
\| p_{j,2}^{i,k} \|_{L^\infty(\mathbf{R}^n)} + \| \nabla_x p_{j,2}^{i,k} \|_{L^\infty(\mathbf{R}^n)} \leq C_\rho \| q_{j,2,\psi}^{i,k} \|_{L^p(V_{\rho})} \leq C_\rho \| v_{2,j} \|_{L^p(U_{\rho,j})}
\]
for some $p>n$. By an interpolation (cf. \cite[Lemma 5]{BGST}) again, we deduce that
\[
\| p_{j,2}^{i,k} \|_{L^\infty(\mathbf{R}^n)} + \| \nabla_x p_{j,2}^{i,k} \|_{L^\infty(\mathbf{R}^n)} \leq C_\rho \| v \|_{vBMO(\Omega)}.
\]
This argument also holds for lower order term 
\[
q_{j,3,\psi}^{i,k} := \bigg( \frac{\partial \eta_k}{\partial x_i} \bigg)_\psi \cdot \left( \sum_{1 \leq l \leq n} (v_{2,j})_{\psi,l} \cdot \bigg( \frac{\partial \eta_n}{\partial x_l} \bigg)_\psi \right) \cdot \frac{\partial^2 x_i}{\partial \eta_k \partial \eta_n}.
\]
By letting $q_{j,3}^{i,k} := q_{j,3,\psi}^{i,k} \circ \psi^{-1}$ in $U_{\rho_0,j}$ and $p_{j,3}^{i,k} := E \ast q_{j,3}^{i,k}$, we can show that 
\[
\| p_{j,3}^{i,k} \|_{L^\infty(\mathbf{R}^n)} + \| \nabla_x p_{j,3}^{i,k} \|_{L^\infty(\mathbf{R}^n)} \leq C_\rho \| v \|_{vBMO(\Omega)}.
\]

Set
\[
p_j^{\mathrm{tan}} := \underset{1 \leq k \leq n-1}{\sum_{1 \leq i \leq n,}} \big( q_{j,1}^{i,k} + p_{j,1}^{i,k} + p_{j,2}^{i,k} + p_{j,3}^{i,k} \big).
\]
Since a direct calculation implies that
\begin{align*}
(\operatorname{div}_x w_j^{\mathrm{tan}})_{\psi} &= \underset{1 \leq k \leq n-1}{\sum_{1 \leq i \leq n,}} \bigg( \frac{\partial \eta_k}{\partial x_i} \bigg)_\psi \cdot \partial_{\eta_k} (v_{2,j})_{\psi,i} \\
&- \underset{1 \leq k \leq n-1}{\sum_{1 \leq i \leq n,}} \bigg( \frac{\partial \eta_k}{\partial x_i} \bigg)_\psi \cdot \left( \sum_{1 \leq l \leq n} (v_{2,j})_{\psi,l} \cdot \bigg( \frac{\partial \eta_n}{\partial x_l} \bigg)_\psi \right)\cdot \frac{\partial^2 x_i}{\partial \eta_k \partial \eta_n}
\end{align*}
in normal coordinate in $V_{4 \rho} = \psi^{-1}(U_{4 \rho,j})$, it is easy to see that
\[
- \Delta_x p_j^{\mathrm{tan}} = \operatorname{div} w_j^{\mathrm{tan}}
\]
in $U_{2 \rho,j}\cap \Omega$.
Calculations above ensures that
\[
[\nabla_x p_j^{\mathrm{tan}}]_{BMO(\mathbf{R}^n)} \leq C_\rho \| v \|_{vBMO(\Omega)}.
\]
Since $\operatorname{supp} q_{j,1}^{i,k} \subset U_{4 \rho,j}$, we consider $x \in \Gamma$ and $r < \rho$ such that $B_r(x) \cap U_{4 \rho,j} \neq \emptyset$. By change of variables $y = \psi(\eta)$ in $U_{4 \rho,j}$, we deduce that
\[
\int_{B_r(x) \cap U_{4 \rho,j}} | \nabla_y q_{j,1}^{i,k} \cdot \nabla_y d| \, dy \leq C \int_{B_{L_\ast r}(\zeta)} | \partial_{\eta_n} q_{j,1,\psi}^{i,k} | \, d\eta
\]
where $\zeta = \psi^{-1}(x)$ and $\zeta_n = 0$. By Lemma \ref{4P}, we see that 
\[
\int_{B_{L_\ast r}(\zeta)} | \partial_{\eta_n} q_{j,1,\psi}^{i,k} | \, d\eta \leq r^n C_\rho \| v \|_{vBMO(\Omega)}.
\]
Since $\nabla_x p_{j,l}^{i,k} \in L^\infty(\mathbf{R}^n)$ for $l = 1,2,3$, we finally obtain that
\[
\frac{1}{r^n} \int_{B_r(x)} | \nabla_y p_j^{\mathrm{tan}} \cdot \nabla_y d | \, dy \leq C_\rho \| v \|_{vBMO(\Omega)}.
\]
\end{proof}

\subsection{Volume potential for normal component} \label{3VPN} 
Consider $\rho \in (0,\rho_\ast/2)$ and $1 \leq j \leq m$. We let $g_j := \nabla d \cdot (\varphi_j v_2)_{\mathrm{odd}}$. Since $\varphi_j v_2 \in vBMO(\Omega)$ with $\operatorname{supp} \varphi_j v_2 \subset U_{\rho,j} \cap \overline{\Omega}$, by Proposition \ref{2E} we see that $g_j \in BMO(\mathbf{R}^n) \cap b^\nu(\Gamma)$. In particular, we have the estimate
\[
[g_j]_{BMO(\mathbf{R}^n)} + [g_j]_{b^\nu(\Gamma)} \leq C_\rho \| v \|_{vBMO(\Omega)}.
\]
Considering the normal coordinate in $U_{4 \rho,j}$, $g_j$ is odd in $\eta_n$. Note that $w_j^{\mathrm{nor}} = g_j \nabla d$.
\begin{proof}[Proposition \ref{2LP}]
Since $\nabla d \in C^1(U_{\rho_0,j})$, by Proposition \ref{2E} we have that
\[
[w_j^{\mathrm{nor}}]_{BMO(\mathbf{R}^n)} \leq C \|\nabla d\|_{C^\gamma(U_{\rho_0,j})}  \|g_j\|_{BMOL^1(\mathbf{R}^n)} \leq C_\rho \| v \|_{vBMO(\Omega)}.
\]
We note that
\[
	\operatorname{div}_x w_j^{\mathrm{nor}} = \nabla_x g_j \cdot \nabla_x d + g_j \Delta_x d.
\]
Let $g_{j,\psi} := g_j \circ \psi$ in $U_{\rho_0,j}$. We may treat $g_{j,\psi}$ as a function in $\mathbf{R}^n$ with value zero outside $V_\rho$. By Proposition \ref{BNCS}, we have that
\[
[g_{j,\psi}]_{BMO(\mathbf{R}^n)} \leq C_\rho \| g_j \|_{BMOL^1(\mathbf{R}^n)}.
\]
In normal coordinate, $\nabla_x g_j \cdot \nabla_x d=\partial_{\eta_n} g_{j,\psi}$.
We introduce the operator $L=L_0+M$ in Proposition \ref{3R}. Since $g_{j,\psi} \in Z_\rho$, we set
\[
	p_{1,j,\psi} := \theta_\rho L^{-1}_0 \partial_{\eta_n} g_{j,\psi} 
\]
where $\theta_\rho$ is the cut-off function of $V_{2 \rho}$ in the proof of Lemma \ref{3P}. $p_{1,j,\psi}$ satisfies all conditions in Lemma \ref{3P}. 
Set $f_{j,\psi} := - M\theta_\rho L^{-1}_0 \partial_{\eta_n} g_{j,\psi}$. We define 
\[
p_{1,j} := p_{1,j,\psi} \circ \psi^{-1}, f_j := f_{j,\psi} \circ \psi^{-1}
\]
in $U_{\rho_0,j}$. Notice that $p_{1,j} \in L^\infty(\mathbf{R}^n)$ and $f_j \in L^p(\mathbf{R}^n)$ with some $p>n$. By Proposition \ref{BNCV},
\[
[\nabla_x p_{1,j}]_{BMO(\mathbf{R}^n)} \leq C_\rho [\nabla_\eta p_{1,j,\psi}]_{BMO(\mathbf{R}^n)} \leq C_\rho [g_{j,\psi}]_{BMO(\mathbf{R}^n)}.
\]

Set
\[
	p_j^{\mathrm{nor}} = p_{1,j} + p_{2,j} + p_{3,j}
\]
with $p_{2,j} = E \ast f_j$ and $p_{3,j}=E \ast (g_j \Delta_x d)$. This $p_j^{\mathrm{nor}}$ satisfies all desired properties required. For lower order terms $p_{2,j}$ and $p_{3,j}$, we have that
\[
\| p_2 \|_{L^\infty(\mathbf{R}^n)} + \| \nabla p_2 \|_{L^\infty(\mathbf{R}^n)} + \| \nabla p_3 \|_{L^\infty(\mathbf{R}^n)} + \| p_3 \|_{L^\infty(\mathbf{R}^n)} \leq C_\rho \| g_j \|_{L^p(\mathbf{R}^n)}
\]
as $E$ and $\nabla_x E$ are both locally integrable. By an interpolation (cf. \cite[Lemma 5]{BGST}), we obtain that
\[
[\nabla_x p_j^{\mathrm{nor}}]_{BMO(\mathbf{R}^n)} \leq C_\rho \| g_j \|_{BMOL^1(\mathbf{R}^n)} \leq C_\rho \| v \|_{vBMO(\Omega)}.
\]
Since $\operatorname{supp} p_{1,j} \subset U_{\rho,j}$, we consider $x \in \Gamma$ and $r<\rho$ such that $B_r(x) \cap U_{\rho,j} \neq \emptyset$. Set $\zeta = \psi^{-1}(x)$ with $\zeta_n = 0$. Consider change of variable $y = \psi(\eta)$ in $U_{4 \rho,j}$, by Lemma \ref{3P} we see that
\[
\int_{B_r(x) \cap U_{\rho,j}} | \nabla_y d \cdot \nabla_y p_{1,j} | \, dy \leq C \int_{B_{L_\ast r}(\zeta)} | \partial_{\eta_n} p_{1,j,\psi} | \, d\eta \leq C_\rho [g_{j,\psi}]_{BMO(\mathbf{R}^n)}.
\]
By the $L^\infty$-estimates of $\nabla_y p_2$ and $\nabla_y p_3$, we get that
\[
\frac{1}{r^n} \int_{B_r(x)} |\nabla_y d \cdot \nabla_y p_j^{\mathrm{nor}}| \, dy \leq C_\rho \| v \|_{vBMO(\Omega)}.
\]
Finally, a simple substitution shows that
\begin{align*}
- \Delta_x p_j^{\mathrm{nor}} = \nabla_x d \cdot \nabla_x g_j - f_j + f_j + g_j \Delta_x d = \operatorname{div}_x w_j^{\mathrm{nor}}
\end{align*}
in $U_{2 \rho}(z_0)\cap\Omega$.
\end{proof}
%

\section{Neumann problem with bounded data} 
\label{NPB}
We consider the Neumann problem for the Laplace equation problem \eqref{1NP} for the Laplace equation.
 If $\Omega$ is a smooth bounded domain, as well-known, for $g\in H^{-1/2}(\Gamma)$, there is a unique (up to constant) weak solution $u \in H^1(\Omega)$ provided that $g$ fulfills the compatibility condition
\begin{equation} \label{NC}
	\int_\Gamma g \, d \mathcal{H}^{n-1} = 0;
\end{equation}
see e.g.\ \cite{LiMa}.
 The main goal of this section is to prove that $\nabla u$ belongs to $vBMO^{\infty,\infty}(\Omega)$ provided that $g \in L^\infty(\Gamma)$.
 In other words, we prove Lemma \ref{EN}.

To prove Lemma \ref{EN}, we represent the solution by using the Neumann-Green function.
Let $N(x,y)$ be the Green function, i.e., a solution $v$ of
\begin{align*} 
	-\Delta_x v &= \delta(x-y) - |\Omega|^{-1}&\text{in}& \quad\Omega \\
	\frac{\partial v}{\partial\mathbf{n}_x} &= 0&\text{on}&\quad\partial\Omega
\end{align*}
for $y\in\Omega$.
 It is easy to see that the solution $u$ of \eqref{1NP} satisfying $\int_\Omega u \, dx=0$ is given as
\[
	u(x) = \int_\Gamma N(x, y) g(y) \, d\mathcal{H}^{n-1} (y).
\]
The function $N$ is decomposed as
\[
	N(x, y) = E(x - y) + h(x, y),
\]
where $h\in C^\infty(\Omega\times\Omega)$ is a milder part.
 We recall $h(x,y)=h(y,x)$ and
\[
	\sup_{x\in\Omega} \int_\Omega \left| \nabla^k_y h(x, y) \right|^{1+\delta} \, dy < \infty
\]
for $k=0,1,2$ with some $\delta>0$;
see \cite[Lemma 3.1]{GGH}.
In particular, by applying the standard $L^p$ estimate for the Neumann problem in the proof of \cite[Lemma 3.1]{GGH} to $\nabla_y h( \cdot,y)$, we can deduce that
\[
	\sup_{x \in \Omega} \int_\Omega\left|\nabla_x \nabla_y h(x,y)\right|^{1+\delta} \, dy<\infty.
\]
Hence, we see that $\nabla_x h(x,\cdot) \in W^{1,1+\delta}(\Omega_y)$. By the trace theorem for Sobolev space $W^{1,1+\delta}(\Omega_y)$, this yields
\begin{equation} \label{HE}
	M_0 := \sup_{x\in\Omega} \int_\Gamma \left|\nabla_x h(x,y)\right|^{1+\delta} \, d\mathcal{H}^{n-1}(y) < \infty.
\end{equation}
We decompose $u$ as
\[
	u(x) = E * (\delta_\Gamma \otimes g)
	+ \int_\Gamma h(x,y) g(y) \, d\mathcal{H}^{n-1}(y) 
	= I + I\!\!I.
\]
The estimate \eqref{HE} yields
\[
	\| \nabla I\!\!I \|_{L^\infty(\Omega)} 
	\leq M_0 \|g\|_{L^\infty(\Gamma)}, 
\]
so to prove Lemma \ref{EN} it suffices to estimate $\nabla I$.
In other words, Lemma \ref{EN} follows from the next lemma.
\begin{lemma} \label{NM}
Let $\Omega$ be a bounded domain in $\mathbf{R}^n$ with $C^2$ boundary $\Gamma=\partial\Omega$.
\begin{enumerate}
\item[(i)] ($BMO$ estimate)
 There exists a constant $C_1$ such that
\begin{equation} \label{BMOE}
	\left[ \nabla \left( E * (\delta_\Gamma \otimes g) \right) \right]_{BMO(\mathbf{R}^n)}
	\leq C_1 \|g\|_{L^\infty(\Gamma)}
\end{equation}
for all $g\in L^\infty(\Gamma)$.
\item[(i\hspace{-1pt}i)] ($L^\infty$ estimate for normal component)
There exists a constant $C_2$ such that
\begin{equation} \label{TE}
	\left\| \nabla d \cdot \nabla \left( E * (\delta_\Gamma \otimes g) \right) \right\|_{L^\infty(\Gamma_{\rho_0}^{\mathbf{R}^n} \cap \Omega)}
	\leq C_2 \|g\|_{L^\infty(\Gamma)}
\end{equation}
for all $g\in L^\infty(\Gamma)$.
\end{enumerate}
\end{lemma}
Here $E*(\delta_\Gamma\otimes g)$ is defined as $E*(\delta_\Gamma\otimes g)(x) := \int_\Gamma E(x-y)g(y) \, d\mathcal{H}^{n-1}(y)$ for a function $g$ on $\Gamma$.
We shall prove Lemma \ref{NM} in following subsections.

\subsection{$BMO$ estimate} \label{BMOES} 
To see the idea, we shall prove \eqref{BMOE} when $\Gamma$ is flat.
 Let $\Gamma=\partial\mathbf{R}^n_+$ and $\mathbf{R}^n_+=\left\{(x_1,\ldots,x_n)\mid x_n>0 \right\}$.
 In this case,
\[
	\nabla \left( E*(\delta_\Gamma\otimes g) \right)
	= \nabla\partial_{x_n} E*1_{\mathbf{R}^n_+} \widetilde{g}
\]
where $\widetilde{g} \in L^\infty(\mathbf{R}^n)$ is defined by $\widetilde{g} (x',x_n) := g(x',0)$ for any $x \in \mathbf{R}^n$.
By the $L^\infty$-$BMO$ estimate for the singular integral operator \cite[Theorem 4.2.7]{GraM}, we obtain \eqref{BMOE} when $\Gamma=\partial\mathbf{R}^n_+$.
%
\begin{proof}[Lemma \ref{NM} (i)]
Note that the signed distance function $d$ is $C^2$ in $\Gamma_{\rho_0}^{\mathbf{R}^n}$, see \cite[Section 14.6]{GT}. Let $\delta \in \rho_0/2$.
We take a $C^2$ cut-off function $\theta\geq 0$ such that $\theta(\sigma)=1$ for $\sigma \leq 1$ and $\theta(\sigma)=0$ for $\sigma \geq 2$.
By the choice of $\delta$, we see that $\theta_d=\theta(d/\delta)$ is $C^2$ in $\mathbf{R}^n$.
We extend $g\in L^\infty(\Gamma)$ to $g_e \in L^\infty(\Gamma_{2 \delta}^{\mathbf{R}^n})$ by setting
\[
g_e(x) := g(\pi x)
\]
for any $x \in \Gamma_{2 \delta}^{\mathbf{R}^n}$ with $\pi x$ denoting the projection of $x$ on $\Gamma$. For $x \in \Gamma_{2 \delta}^{\mathbf{R}^n}$, by considering the normal coordinate $x = \psi(\eta)$ in $U_{2 \delta}(\pi x)$, we have that
\[
(\nabla_x d)_\psi \cdot (\nabla_x g_e)_\psi = \partial_{\eta_n} (g_e)_\psi = 0
\]
as $(g_e)_\psi(\eta',\alpha) = (g_e)_\psi(\eta',\beta)$ for any $|\eta'| < 2 \delta$ and $\alpha,\beta \in (-2 \delta,2 \delta)$.
Hence, we see that $\nabla d\cdot\nabla g_e=0$ in $\Gamma_{2 \delta}^{\mathbf{R}^n}$.

Let us consider $g_{e,c} := \theta_d g_e$. A key observation is that
\begin{align*}	
	\delta_\Gamma \otimes g &= (\nabla 1_\Omega \cdot \nabla d) g_{e,c} \\
	&= \operatorname{div} (g_{e,c} 1_\Omega \nabla d) - 1_\Omega \operatorname{div} (g_{e,c} \nabla d), \\
	\operatorname{div} (g_{e,c} \nabla d) &= g_{e,c} \Delta d + \nabla d \cdot \nabla g_{e,c} = g_{e,c} \Delta d + \frac{\theta'(d/\delta)}{\delta} g_e.
\end{align*}
Thus
\[
	\nabla E*(\delta_\Gamma \otimes g) = \nabla\operatorname{div} \left(E*( g_{e,c} 1_\Omega \nabla d) \right)
	-\nabla E* \left( 1_\Omega g_e f_{\theta,\delta} \right) = I_1 + I_2
\]
where $f_{\theta,\delta} := \theta_d \Delta d + \frac{\theta'(d/\delta)}{\delta}$. 
By the $L^\infty$-$BMO$ estimate for the singular integral operator \cite[Theorem 4.2.7]{GraM}, the first term is estimated as
\[
	[I_1]_{BMO(\mathbf{R}^n)} \leq C \| g_{e,c} \nabla d \|_{L^\infty(\Omega)} \leq C \|g\|_{L^\infty(\Gamma)}.
\]
Since
\[
	A = \sup_{x \in \mathbf{R}^n \setminus \{0\}} |x|^{n-1} \left| \nabla E(x) \right| < \infty,
\]
for $x \in \mathbf{R}^n$ with $d(x,\Omega) = \inf_{y \in \Omega} |x-y| < 1$ we have that
\[
	\left| I_2(x) \right| \leq A \int_\Omega \frac{1}{|x-y|^{n-1}} \, dy \| f_{\theta,\delta} \|_{L^\infty(\Gamma_{2 \delta}^{\mathbf{R}^n})} \| g_{e,c} \|_{L^\infty(\Gamma_{2 \delta}^{\mathbf{R}^n})} \leq C_{\Omega,\delta} \|g\|_{L^\infty(\Gamma)}
\]
with $C_{\Omega,\delta}$ depending only on $\Omega$ and $\delta$.
For $x \in \mathbf{R}^n$ with $d(x,\Omega) = \inf_{y \in \Omega} |x-y| \geq 1$, the above estimate is trivial as $|x-y|^{-(n-1)} \leq 1$ for any $y \in \Omega$.
The proof of (i) is now complete.
\end{proof}

\subsection{Estimate for normal derivative} \label{END} 

We shall estimate normal derivative of $E$.
\begin{lemma} \label{EP} 
Let $\Omega$ be a bounded domain in $\mathbf{R}^n$ with $C^2$ boundary $\Gamma$.
 Then
\begin{enumerate}
\item[(i)]
\[
	\int_\Gamma \frac{\partial E}{\partial\mathbf{n}_y} (x-y) \, d\mathcal{H}^{n-1}(y) = -1
	\quad\text{for}\quad x \in \Omega,
\]
\item[(i\hspace{-1pt}i)]
\[
	\sup_{x\in\Omega} \int_\Gamma \left| \frac{\partial E}{\partial\mathbf{n}_y} (x-y) \right| \, d\mathcal{H}^{n-1}(y) < \infty.
\]
\end{enumerate}
\end{lemma}
\begin{proof}
\begin{enumerate}
\item[(i)] This follows from the Gauss divergence theorem.
 We observe that
\[
	\int_\Gamma \frac{\partial E}{\partial\mathbf{n}_y} (x-y) \, d\mathcal{H}^{n-1}(y) = \int_\Omega \Delta_y E(x-y) \, dy.
\]
Since $\Delta_y E(x-y) = -\delta(x-y)$, we obtain
\[
	\int_\Gamma \frac{\partial E}{\partial\mathbf{n}_y} (x-y) \, d\mathcal{H}^{n-1}(y) = -1
\]
for $x\in\Omega$.
\item[(i\hspace{-1pt}i)] We recall our local coordinate patches $\{U_i\}^m_{i=1}$ with $U_i=U_{\rho,i}$ as in Section \ref{LOC}.
For $x \in \Omega^\rho$ and $y \in \Gamma$, obviously $| \nabla E (x-y) | \leq C \rho^{-(n-1)}$. Let $x \in \Gamma_\rho^{\mathbf{R}^n} \cap \Omega$. If $d(x, U_i \cap \Gamma) \geq \rho$, similarly $| \nabla E (x-y) | \leq C \rho^{-(n-1)}$ for $y \in U_i \cap \Gamma$. Hence, it is sufficient to consider $U_i$ such that $d(x, U_i \cap \Gamma) < \rho$, i.e., it suffices to prove
\begin{equation*} \label{EL}
	\int_{U_i\cap\Gamma} \left| \frac{\partial E}{\partial\mathbf{n}_y}(x-y) \right| d\mathcal{H}^{n-1}(y) < \infty.
\end{equation*}
for $U_i$ such that $d(x, U_i \cap \Gamma) < \rho$.
Since $-\partial E/\partial\mathbf{n}_y(x-y)$ is invariant under translations and rotations, we can write $-\partial E/\partial\mathbf{n}_y(x-y)$ in the local coordinate. Let $U_i$ be such that $d(x, U_i \cap \Gamma) < \rho$ and denote $h_{z_i}$ by $h_i$ for simplicity. Let us observe that
\[
	-\mathbf{n} \left( y', h_i(y') \right) = \left( -\nabla' h_i(y'), 1 \right)/ \omega_i(y') 
\]
with $\omega_i(y')=\left( 1+ |\nabla' h_i(y')|^2 \right)^{1/2}$, where $\nabla'$ is the gradient in $y'$ variables.
 This implies that 
\[
	-n\alpha(n)\frac{\partial E}{\partial\mathbf{n}_y}(x-y) 
	= \frac{\sigma_i(y')}{\omega_i(y') \left( |x'-y'|^2 + \left( x_n- h_i(y') \right)^2 \right)^{n/2}}
\]
for $y \in \Gamma_i$ with
\[
	\sigma_i(y') := -\nabla' h_i(y) \cdot (x'-y') + \left(x_n - h_i(y') \right)
	\; \; \text{where} \; \; x_n > h_i(x'),\ 
	x' \in B_{3 \rho}(0').
\]
We set
\[
	K_i(x', y', x_n) = \frac{\sigma_i(y')}{\left( |x'-y'|^2 + \left( x_n - h_i(y') \right)^2 \right)^{n/2}}.
\]
By the Taylor expansion
\[
	h_i(x') = h_i(y') + \nabla' h_i(y') \cdot (x'-y') + r_i(x',y')
\]
with
\[
	r_i(x',y') = (x'-y')^{\mathrm{T}} \cdot \int^1_0 (1-\theta) \nabla'^2 h_i \left(\theta x'+(1-\theta)y' \right)d\theta \cdot (x'-y'),
\]
we obtain
\[
	\sigma_i(y') = x_n - h_i(x') + r_i(x',y')
\]
with an estimate
\begin{equation} \label{R1}
	\left| r_i(x',y') \right| \leq \| \nabla'^2 h_i \|_{L^\infty(B_{3 \rho}(0'))} |x'-y'|^2.
\end{equation}
We decompose $K_i$ into a leading term and a remainder term
\[
	K_i(x', y', x_n) = K_0^i(x', y', x_n) + R_i(x', y', x_n)
\]
with
\begin{align*}
	K_0^i(x', y', x_n) &:= \frac{x_n - h_i(x')}{\left( |x'-y'|^2 + \left( x_n - h_i(y') \right)^2 \right)^{n/2}} \\
	R_i(x', y', x_n) &:= \frac{r_i(x, y)}{\left( |x'-y'|^2 + \left( x_n - h_i(y') \right)^2 \right)^{n/2}}.
\end{align*}
The term $K_0^i$ is very singular but it is positive.
The term $R_i$ is estimated as
\[
	\left| R_i(x', y', x_n) \right| \leq \| \nabla'^2 h_i\|_{L^\infty(B_{3 \rho}(0'))} |x'-y'|^{2-n}
\]
by the estimate \eqref{R1}. 
Hence, 
\[
\int_{\Gamma \cap U_i} \left| \frac{R_i(x',y',x_n)}{\omega_i(y')} \right| \, d\mathcal{H}^{n-1}(y) \leq C \int_{B_\rho(0')} \frac{1}{|x'-y'|^{n-2}} \, dy' \leq C \rho
\]
with $C$ independent of $\rho$ and $i$.
By (i), we observe that
\begin{align*}
	n \alpha(n) 
	&= \sum_{i : d(x, U_i \cap \Gamma) < \rho} \int_{B_\rho(0')} \frac{K_i(x', y', x_n)}{\omega_i(y')} \, dy' \\
	&- n \alpha(n) \sum_{j : d(x, U_j \cap \Gamma) \geq \rho} \int_{U_j \cap \Gamma} \frac{\partial E}{\partial\mathbf{n}_y} (x-y) \, d\mathcal{H}^{n-1}(y).
\end{align*}
Since $K_0^i$ is positive for any $i$ such that $d(x,U_i \cap \Gamma) < \rho$,
\[
\sum_{i : d(x, U_i \cap \Gamma) < \rho} \int_{B_\rho(0')} \frac{K_0^i(x', y', x_n)}{\omega_i(y')} \, dy' \leq n \alpha(n) \cdot (1+\frac{m \cdot C \cdot S(\Gamma)}{\rho^{n-1}}) + m \cdot C \cdot \rho
\]
where $S(\Gamma)$ denotes the surface area of $\Gamma$, which is bounded. Thus, the estimate
\[
\int_{U_i\cap\Gamma} \left| \frac{\partial E}{\partial\mathbf{n}_y}(x-y) \right| \, d\mathcal{H}^{n-1}(y)
\leq \frac{1}{n \alpha(n)} \int_{B_\rho(0')} \frac{K_0^i + |R_i|}{\omega_i(y')} \, dy' < \infty
\]
holds for any $U_i$ such that $d(x, U_i \cap \Gamma) < \rho$.
The proof of (i\hspace{-1pt}i) is now complete.
\end{enumerate}
\end{proof}
%
Based on Lemma \ref{EP}, we are able to prove Lemma \ref{NM} (i\hspace{-1pt}i).
\begin{proof}[Lemma \ref{NM} (i\hspace{-1pt}i)]
We decompose
\begin{multline*}
	\nabla d (x) \cdot \nabla \left(E*(\delta_\Gamma\otimes g) \right)(x)
	= \int_\Gamma \left(\nabla d(x)-\nabla d(y) \right)\cdot \nabla E(x-y) g(y) \, d\mathcal{H}^{n-1}(y) \\
	+ \int_\Gamma \frac{\partial E}{\partial\mathbf{n}_y}(x-y)g(y) \, d\mathcal{H}^{n-1}(y) = I_1 + I_2.
\end{multline*}
Let $x \in \Gamma_{\rho_0}^{\mathbf{R}^n}$ and $\pi x$ be the projection of $x$ on $\Gamma$. For $y \in U_{\rho_0}(\pi x)$, there exists a constant $L'$, independent of $x$ and $y$, such that
\[
	\left| \nabla d(x) - \nabla d(y) \right| \leq L' |x-y|.
\]
For $y \in \Gamma_{\rho_0}^{\mathbf{R}^n} \setminus U_{\rho_0}(\pi x)$, we have that $|x-y| \geq \frac{\rho_0}{2}$. Since $\overline{\Gamma_{\rho_0/2}^{\mathbf{R}^n}}$ is compact in $\mathbf{R}^n$, by considering a finite subcover of $\cup_{z \in \Gamma} U_{\rho_0}(z)$ we are able to show that there exists $M>0$ such that the estimate
\[
|\nabla d (x) - \nabla d (y)| \leq M |x-y|
\]
holds for any $x,y \in \Gamma_{\rho_0}^{\mathbf{R}^n}$.
Thus,
\[
	H(x,y) = \left( \nabla d(x) - \nabla d(y) \right) \cdot \nabla E(x-y)
\]
is estimated as
\[
	\left| H(x,y) \right| \leq \frac{M}{|x-y|^{n-2}}
\]
in $\Gamma_{\rho_0}^{\mathbf{R}^n} \times \Gamma_{\rho_0}^{\mathbf{R}^n}$.
We observe that
\begin{align*}
	\sup_{x \in \Gamma_{\rho_0}^{\mathbf{R}^n} \cap \Omega} \left| I_1(x) \right| &\leq \sup_{x \in \Gamma_{\rho_0}^{\mathbf{R}^n} \cap \Omega} \int_\Gamma H(x,y) \, d\mathcal{H}^{n-1}(y) \|g\|_{L^\infty(\Gamma)} \\
	& \leq M\sup_{x \in \Gamma_{\rho_0}^{\mathbf{R}^n} \cap \Omega} \int_\Gamma \frac{d\mathcal{H}^{n-1}(y)}{|x-y|^{n-2}} \|g\|_{L^\infty(\Gamma)}.
\end{align*}
Since
\[
	\sup_{x \in \Gamma_{\rho_0}^{\mathbf{R}^n} \cap \Omega} \left| I_2(x) \right| \leq \sup_{x \in \Gamma_{\rho_0}^{\mathbf{R}^n} \cap \Omega} \int_\Gamma \left| \frac{\partial E}{\partial\mathbf{n}_y}(x-y) \right| \, d\mathcal{H}^{n-1}(y) \|g\|_{L^\infty(\Gamma)},
\]
Lemma \ref{EP} (i\hspace{-1pt}i) now yields \eqref{TE}.
The proof is now complete.
\end{proof}
%
We wonder whether the tangential component of $\nabla E*(\delta_\Gamma\otimes g)$ satisfies the same estimate.
Unfortunately, the estimate
\[
	\left\| \nabla \left(E*(\delta_\Gamma\otimes g)\right) \right\|_{L^\infty(\Gamma_{\rho_0}^{\mathbf{R}^n} \cap \Omega)} \leq C \|g\|_{L^\infty(\Gamma)}
\]
should not hold even if $\Gamma$ is flat.
 Even weaker estimate
\[
	\left[ \nabla \left(E*(\delta_\Gamma\otimes g)\right) \right]_{b^\nu(\Gamma)} \leq C \|g\|_{L^\infty(\Gamma)}
\]
should not hold in general.

To illustrate the problem, we consider the case that $\Gamma$ is flat.
 We may assume $\Gamma=\partial\mathbf{R}^n_+$, $\mathbf{R}^n_+ = \{x_n>0 \}$.
\begin{lemma} \label{FP}
The estimate
\[
	\left\| \partial_{x_n} \left(E*(\delta_\Gamma\otimes g) \right) \right\|_{L^\infty(\mathbf{R}_+^n)} \leq \frac{1}{2}\|g\|_{L^\infty(\mathbf{R}^{n-1})}
\]
holds for $g \in L^\infty(\mathbf{R}^{n-1})$.
\end{lemma}
\begin{proof}
This is because $-\partial_{x_n} \left(E*(\delta_\Gamma\otimes g) \right)$ is the half of the Poisson integral, i.e.,
\[
	-\partial_{x_n} \left(E*(\delta_\Gamma\otimes g) \right)(x)
	= \frac{1}{2} \int_{\mathbf{R}^{n-1}} P_{x_n} (x'-y') g(y')dy',
\]
where $P_{x_n}$ denotes the Poisson kernel.
 Thus the desired $L^\infty$ estimate follows from the maximum principle of the Dirichlet problem for the Laplacian or from the property that $\int_{\mathbf{R}^{n-1}} P_{x_n} (x')dx'=1$ and $P_{x_n} \geq 0$.
\end{proof}
%
%
\begin{theorem} \label{CE}
There is a bounded sequence of smooth functions $\{g_\ell\}_{\ell \in \mathbf{N}} \subset L^\infty(\mathbf{R}^{n-1})$ such that
\[
	\lim_{\ell\to\infty} \left[ \partial_{x'} \left(E*(\delta_\Gamma\otimes g_\ell) \right) \right]_{b^\nu} = \infty
\]
for any $\nu>0$.
\end{theorem}
\begin{proof}
If $g$ is smooth, $E*(\delta_\Gamma\otimes g)$ is smooth up to the boundary.
 In this case, if $\left[\partial_{x'}\left(E*(\delta_\Gamma\otimes g)\right)\right]_{b^\nu}$ is bounded by $C\|g\|_{L^\infty(\mathbf{R}^{n-1})}$, $\left\| \partial_{x'} \left(E*(\delta_\Gamma\otimes g) \right) \right\|_{L^\infty(\Gamma)}$ is also bounded by $c_0 C\|g\|_{L^\infty(\mathbf{R}^{n-1})}$ with a constant $c_0$ depending only on $n$ since the mean value over $r$-ball around $x$ converges to its value at $x$ as $r\to 0$.

We consider the Neumann problem
\begin{align*}
	\Delta u = 0	&\quad\text{in}\quad \mathbf{R}^n_+, \\
	\frac{\partial u}{\partial\mathbf{n}} = g &\quad\text{on}\quad \Gamma = \partial\mathbf{R}^n_+.
\end{align*}
By using the tangential Fourier transform, we see that
\[
	u(x,t) = \Lambda^{-1} \exp(-x_n \Lambda) g
\]
where $\Lambda=(-\Delta')^{1/2}$.
 If $\|\nabla' u\|_{L^\infty(\Gamma)} \leq C\|g\|_{L^\infty(\mathbf{R}^{n-1})}$ were true, sending $x_n>0$ to zero would imply $L^\infty$ boundedness of the Riesz operator $\nabla' \Lambda^{-1}$, which is absurd.

The operator $E*(\delta_\Gamma\otimes g)$ is the half of the solution operator of the Neumann problem, so $L^\infty$ bound for $\nabla' E*(\delta_\Gamma\otimes g)$ should not hold even if it is restricted to smooth functions.
\end{proof}
\begin{corollary} \label{4UB}
Assume that $\Omega=\mathbf{R}^n_+$.
Let $v\mapsto\nabla q$ be the Helmholtz projection to a gradient field.
Then, this projection is unbounded from $\left(L^\infty(\Omega)\right)^n$ to $\left(BMO^{\mu,\nu}_b(\Omega)\right)^n$ for any $\mu,\nu>0$.
\end{corollary}

\begin{proof}
We consider
\[
	v = \left(0, \ldots, 0, v_n(x')\right)
\]
with $v_n \in L^\infty(\mathbf{R}^{n-1})$.
 This evidently solves $\operatorname{div}v=0$.
 The normal trace equals $-v_n(x')$.
 If
\[
	[\nabla q]_{b^\nu} \leq C \|v_n\|_{L^\infty(\mathbf{R}^{n-1})}
\]
for all $v_n \in L^\infty(\mathbf{R}^{n-1})$ with $C$ independent of $v$, then this would contradict Theorem \ref{CE}.
\end{proof}
\section*{Acknowledgement}
The first author was partly supported by the Japan Society for the Promotion of Science through grants No. 19H00639 (Kiban A), No. 18H05323 (Kaitaku), No. 17H01091 (Kiban A). The second author is partly supported by the Mizuho international foundation through foreign students scholarship program.
%

%
%



\end{document}